\documentclass[a4paper, 10pt, oneside]{amsart}

\usepackage[dvipsnames]{xcolor}
\usepackage{latexsym,amssymb}
\usepackage[english]{babel}
\usepackage[utf8]{inputenc}
\usepackage{amsmath,amsthm,bm}
\usepackage[T1]{fontenc}
\usepackage{geometry}
\usepackage{color}
\usepackage[colorlinks,pdfpagelabels,pdfstartview = FitH,bookmarksopen
= true,bookmarksnumbered = true,linkcolor = blue,plainpages =
false,hypertexnames = false,citecolor = red,pagebackref=false]{hyperref}

\usepackage{amsbsy}
\usepackage{amstext}
\usepackage{amssymb}
\usepackage{esint}
\usepackage{stmaryrd}
\usepackage{graphicx}

\allowdisplaybreaks
\newtheorem{theorem}{Theorem}
\newtheorem{lemma}[theorem]{Lemma}
\newtheorem{definition}[theorem]{Definition}

\newtheorem{corollary}[theorem]{Corollary}

\theoremstyle{definition}
\newtheorem{remark}[theorem]{Remark}
\numberwithin{theorem}{section}
\numberwithin{equation}{section}

\def\N{\mathbb{N}}
\def\R{\mathbb{R}}

\renewcommand{\d}{\:\! \mathrm{d}}
\newcommand{\dx}{\mathrm{d}x}

\newcommand{\dt}{\mathrm{d}t}
\newcommand{\ds}{\mathrm{d}s}

\renewcommand{\epsilon}{\varepsilon}

\DeclareMathOperator{\Div}{div}

\DeclareMathOperator{\dist}{dist}
\DeclareMathOperator{\loc}{loc}

\renewcommand{\epsilon}{\varepsilon}
\newcommand{\eps}{\varepsilon}
\renewcommand{\rho}{\varrho}

\def\eqn#1$$#2$${\begin{equation}\label#1#2\end{equation}}

\newcommand{\mollifytime}[2]{[\![ #1 ]\!]_{#2}}
\newcommand{\norm}[1]{\lVert \kern 0.6pt #1 \kern 0.6pt \rVert}

\def\Xint#1{\mathchoice
    {\XXint\displaystyle\textstyle{#1}}%
    {\XXint\textstyle\scriptstyle{#1}}%
    {\XXint\scriptstyle\scriptscriptstyle{#1}}%
    {\XXint\scriptscriptstyle\scriptscriptstyle{#1}}%
    \!\int}
\def\XXint#1#2#3{\setbox0=\hbox{$#1{#2#3}{\int}$}
    \vcenter{\hbox{$#2#3$}}\kern-0.5\wd0}

\def\dashint{\Xint{\raise4pt\hbox to7pt{\hrulefill}}}

\def\XXiint#1#2#3{\setbox0=\hbox{$#1{#2#3}{\iint}$}
    \vcenter{\hbox{$#2#3$}}\kern-0.5\wd0}

\renewcommand{\u}{\boldsymbol{u}}
\renewcommand{\v}{\mathbf{v}}
\renewcommand{\a}{\boldsymbol{a}}
\newcommand{\g}{\boldsymbol{g}}
\newcommand{\A}{\mathbf{A}}


\author[K. Moring]{Kristian Moring}
\address{Kristian Moring\\
Department of Mathematics and Systems Analysis, Aalto University\\
P.~O.~Box 11100, FI-00076 Aalto, Finland}
\email{kristian.moring@aalto.fi}

\author[R. Rainer]{Rudolf 
Rainer}
\address{Rudolf Rainer\\
Fachbereich Mathematik, Universit\"at Salzburg \\
Hellbrunner Str. 34, 5020 Salzburg, Austria}
\email{rudolf.rainer@sbg.ac.at}

\begin{document}
\title{Stability for systems of porous medium type}



\begin{abstract}
We establish stability properties of weak solutions for systems of porous medium type with respect to the exponent $m$. Thereby we treat stability for the local case as well as for Cauchy-Dirichlet problems. Both degenerate and singular cases are covered. \\[2ex]
\end{abstract}

\maketitle

\section{Introduction}

 

\noindent
Point of interest is the stability of weak solutions to parabolic systems
\begin{align}\label{intro-eq}
  \partial_t u -\Div \mathbf A \big(x,t,u,D(|u|^{m-1} u )\big) = 0
\end{align}
in a cylindrical domain with respect to the exponent $m$. $\A$ is a vector field whose structural properties are detailed further down. This general type is labelled system of the porous medium type, as it contains as its principal prototype the porous medium equation
\begin{align*}
\partial_t u -\Delta( u^m ) = 0.
\end{align*}

The equation is divided into two regimes: If $0<m<1$ one speaks of the singular or also fast diffusion case, while for $m >1$ one speaks of the degenerate or slow diffusion case. Both cases will be treated, although we will have a restriction in the singular case. In particular, a positive lower bound for $m$ is required. This matches up with regularity results for the porous medium equation, as the same bound appears e.g. in \cite{Boegelein-Duzaar-Scheven:2018} and~\cite[§6.21]{DBGV-book}.

We will answer the question whether weak solutions of \eqref{intro-eq} converge to a solution of the limit problem as the exponent $m$ varies. This ensures that the solutions of the equation are stable under small perturbations of the parameter $m$, which in applications may be known only approximately. In the first part, local convergence will be studied. We assume weak convergence of the sequence of solutions in this case in order to identify the limit. In the second part, we inspect a Cauchy-Dirichlet problem, where the solutions are expected to attain given initial and boundary values.

For the parabolic $p$-Laplace equation the stability question has been treated by Kinnunen and Parviainen in \cite{Kinnunen-Parviainen}. Two ingredients were essential for the proof: for one, the lateral boundary of the cylinder must be sufficiently regular. Furthermore, to overcome the difficulty that weak solutions (to the parabolic $p$-Laplace equation) for different exponents are in different parabolic Sobolev spaces, a global higher integrability result is essential. Somewhat surprisingly, neither of these were needed to complete the proof when considering an equation of the type \eqref{intro-eq}. This could stem from the fact that, in contrast to the parabolic $p$-Laplace equation, the spaces in the porous medium setting are fixed, even when the exponents differ. Even though not needed, the higher integrability can still be applied to obtain better convergence properties for the sequence of solutions and their gradients.

For the proof of the local result, we proceed as follows: By Caccioppoli type estimates we obtain a uniform bound on the norms of the solutions in a reflexive Banach space, which in turn implies the weak convergence of a subsequence. To improve the convergence for the solutions from weak to strong, we use a dual pairing argument which then allows us to use the compactness properties of parabolic Sobolev spaces, more specifically Theorem 3 in \cite{Simon}. To improve the convergence for the gradients, in \cite{Kinnunen-Parviainen} the authors showed that they form a Cauchy sequence in order to avoid testing with the limit function itself. In this case, we are able to show it directly.

In the global case, we apply the local result. It remains to extend the obtained convergences from local to global, which we do by applying a measure theoretic argument: One can observe that strong convergence in $L^1$ or even pointwise a.e. convergence together with boundedness in $L^2$ implies strong convergence in $L^q$ for all $q<2$. We conduct the argument in detail in Lemma \ref{lem:strong-lp-convergence} and then reuse it several times throughout the proofs.

We shall give a brief recap of the recent history in the research of stability questions. Lindqvist studied stability questions for the stationary $p$-Laplace equation in \cite{Lindqvist}, already in 1987. Due to the mentioned difficulties arising from varying Sobolev spaces, the stability problem for parabolic $p$-Laplacian was settled only after higher integrability was proven. First, Kinnunen and Lewis showed the local higher integrability in \cite{Kinnunen-Lewis:1}, which was then extended up to the boundary by Parviainen \cite{Parviainen} in 2009. This allowed Kinnunen and Parviainen to prove the stability for the parabolic $p$-Laplacian \cite{Kinnunen-Parviainen} one year later. Lukkari and Parviainen studied similar stability questions for the parabolic $p$-Laplace in the degenerate case in \cite{Lukkari-Parviainen}. They also took into account measure data at the initial boundary. Regarding equations of the porous medium type, Lukkari inspected nonnegative weak solutions to the model equation in \cite{Lukkari}. He used the specific structure of the model equation, which is not available in our general setting. \\
Further, in \cite{Benilan-Crandall} the theory of nonlinear semigroups is applied to obtain a stability result for an initial-value problem for equations of the form $\partial_t u - \Delta \varphi (u) = 0$ with a non-linearity $\varphi$. By applying the \lq\lq doubling of variables\rq\rq~method of Kruzkov, quantitative stability estimates in the sense of continuous dependencies and error estimates are obtained in \cite{Chen-Karlsen,Cockburn-Gripenberg,Karlsen-Risebro}. \\
Additionally, we mention the following border cases: For stability results for the case $m \rightarrow \infty$, where the limit problem is sometimes termed the mesa problem, we refer to \cite{Benilan-Boccardo-Herrero,Benilan-Igbida,Caffarelli-Friedman}. For $m \rightarrow 0$, where the limit problem is $\partial_t u - \Delta \log u = 0$, we refer to \cite{Dibenedetto-Gianazza-Liao,Hui}. Also worth noting is \cite{Hui-veryfast}, where the limit $m \rightarrow 0^-$ is inspected, so considering the very fast diffusion equation with $m<0$.


\medskip

\noindent
{\bf Acknowledgments.} K.~Moring has been supported by the Magnus Ehrnrooth Foundation. R. Rainer has been supported by the FWF-Project P 31956 ``Doubly Nonlinear Evolution Equations''.

\section{Preliminaries}

\subsection{Statement of the local result}
\label{sec:statement} 
\hfill \\[1ex]
We consider porous medium systems of the type 
\begin{equation}\label{por-med-eq}
  \partial_t u -\Div \mathbf A(x,t,u,D\u^m) = 0 \quad \text{ in } \Omega_T,
\end{equation}  
in which $\Omega_T := \Omega \times (0,T)$ is a space-time cylinder. $\Omega$ is a bounded open subset of $\R^n$ and $T>0$. We consider $n \geq 2$ and use the abbreviation $\u^m := |u|^{m-1}u$. $\partial_t$ denotes the time derivative, while $D=D_x$ and $\mathrm{div} = \mathrm{div}_x$ denote the derivatives and the divergence with respect to the spatial variable $x$. For open sets $A,B \subset \mathbb R^{n+1}$, we write $A \Subset B$ if $\overline A$ is a compact subset of $B$.

The assumptions on the vector field $\mathbf A\colon \Omega_T\times\R^N\times
\R^{Nn}\to \R^{Nn}$ are as follows. We assume that $\mathbf A$ is a Carathéodory function, i.e. it is measurable with respect to $(x, t) \in \Omega_T$ for all $(u,\xi)\in \R^N\times\R^{Nn}$ and continuous with respect to $(u,\xi)$ for a.e.~$(x,t)\in\Omega_T$. Moreover, we assume that $\mathbf A$ satisfies the following structural conditions with
$0<\nu\le L<\infty$:
\begin{equation}\label{growth}
\left\{
\begin{array}{c}
	\mathbf A(x,t,u,\xi)\cdot\xi \ge \nu|\xi|^2\, ,\\[6pt]
	| \mathbf A(x,t,u,\xi)|\le L|\xi|,
\end{array}
\right.
\end{equation} 
for a.e.~$(x,t)\in \Omega_T$ and any $(u,\xi)\in \R^N\times\R^{Nn}$. We also assume that the vector field is monotone in the sense that for some $\mu \in (0,\infty)$,
\begin{equation} \label{monotone}
\big( \mathbf A(x,t,u,\xi) - \mathbf A(x,t,v,\eta) \big) \cdot (\xi - \eta) \geq \mu |\xi - \eta|^2
\end{equation}
holds true for a.e.~$(x,t)\in \Omega_T$ and for any pairs $(u,\xi), (v,\eta)\in \R^N\times\R^{Nn}$. We work with weak solutions, which we define now.

\begin{definition}\label{def:weak_solution}\upshape
Assume that the vector field $\mathbf A\colon \Omega_T\times \R^N\times\R^{Nn}\to\R^{Nn}$ satisfies \eqref{growth} and \eqref{monotone}. 
We identify a measurable map
$u\colon\Omega_T\to\R^N$ in the class
\begin{equation*}
	\u^m\in L^2_{\loc}\big(0,T;W^{1,2}_{\rm loc}(\Omega,\R^N)\big),
\end{equation*} 
with additional assumption $u \in L^{m+1}_{\loc}(\Omega_T,\R^N)$ in case $m < 1$, as a \textit{weak solution} to the porous medium type system \eqref{por-med-eq} with exponent $m$ if and only if the identity
\begin{align}\label{weak-solution}
	\iint_{\Omega_T}\big[- u\cdot\partial_t\varphi + \mathbf A(x,t,u,D\u^m)\cdot D\varphi\big]\dx\dt
    	= 0
\end{align}
holds true for any testing function $\varphi\in
C_0^\infty(\Omega_T,\R^N)$. 
\qed
\end{definition}

The assumptions on $\varphi$ can be weakened. It suffices that 
\begin{align*}
\varphi \in W^{1,2}(0,T; L^2(\Omega, \mathbb R^N))
	\cap L^2(0,T;W^{1,2}_0(\Omega, \mathbb R^N))
\end{align*}
and $\text{supp } \varphi \Subset \Omega_T$ when $m\geq 1$.
If $m<1$, we further demand $\partial_t \varphi \in L^{\frac{1+m}{m}}(\Omega_T, \mathbb R^N)$ to ensure the finiteness of the integral of the parabolic part of the equation.

\begin{remark} \label{rem:cont_representative_loc}
In Section~\ref{section:loc_continuity} we will prove that a weak solution $u$ according to Definition~\ref{def:weak_solution} has a representative in class $C((0,T); L^{m+1}_{\loc}(\Omega,\R^N))$.
\end{remark}

We denote the critical exponent by $m_c := \frac{(n-2)_+}{n+2}$, where $(n-2)_+ := \max\lbrace n-2,0 \rbrace$. Let $(m_i)$ be a sequence of real numbers in $(m_c,\infty)$ such that $m_i \longrightarrow m \in (m_c,\infty)$ as $i \longrightarrow \infty$. Let further $u_i$ be a weak solution to the Equation~\eqref{por-med-eq} with exponent $m_i$. We assume that there exists a measurable function $u: \Omega_T \to \R^N$, such that as $i \to \infty$,
%
%
\begin{equation}\label{assumption:weak-convergence}
	\u_i^{m_i} \rightharpoonup \u^{m}\quad \text{weakly in } L^2_{\loc}(\Omega_T,\R^N).
\end{equation} 
Moreover, if $m \in (m_c,1)$, we make an additional assumption, namely
\begin{equation} \label{assumption:boundedness}
\u_i^{m_i+1} \text{ is bounded in } L^1_{\loc}(\Omega_T,\R^N).
\end{equation}

\noindent
The following is our main result in the local setting. 

\begin{theorem} \label{theorem}
Let $(m_i)_{i \in \N}$ be a sequence in $(m_c,\infty)$ such that $m_i \longrightarrow m \in (m_c,\infty)$ as $i \longrightarrow \infty$. Let $u_i$ be a weak solution of Equation~\eqref{por-med-eq} with exponent $m_i$ in the sense of Definition~\ref{def:weak_solution}, where the vector field $\mathbf A$ satisfies the growth and monotonicity conditions~\eqref{growth} and~\eqref{monotone}. Furthermore, assume that the assumptions~\eqref{assumption:weak-convergence} and~\eqref{assumption:boundedness} are in force. 
Then, for the function $u$ from \eqref{assumption:weak-convergence}, we have $\u^m \in L^{2}_{\mathrm{loc}}(0,T; W^{1,2}_{\mathrm{loc}}(\Omega,\R^N))$ with
\begin{align*}
\u_i^{m_i} \stackrel{i \rightarrow \infty}\longrightarrow \u^m \quad \text{ in } L^{2}_{\mathrm{loc}}(0,T; W^{1,2}_{\mathrm{loc}}(\Omega,\R^N)).
\end{align*}
Moreover, the limit function $u$ is a weak solution to the Equation~\eqref{por-med-eq} with exponent $m$.

\end{theorem}

\subsection{Cauchy-Dirichlet problem}
\hfill \\[1ex]
We further consider stability for a Cauchy-Dirichlet problem of the form
\begin{align}
\label{equation:global-PME}
\left\{
\begin{array}{ll}
\partial_t u- \Div \A(x,t,u, D\u^m) = 0 &\text{ in } \Omega_T, \vspace{1mm}\\
u=g & \text{ on } \partial_{\mathrm{par}}\Omega_T,
\end{array}
\right. 
\end{align}
where $\partial_{\mathrm{par}}\Omega_T := \big( \partial \Omega \times (0,T) \big) \cup \big( \overline{\Omega} \times \lbrace 0 \rbrace \big)$ is the parabolic boundary of $\Omega_T$. Let $m = \lim_{i \to \infty} m_i \in (m_c,\infty)$ as before. In the following we use the shorthand notation
$$
I(u,g) := I_{m}(u,g) := \tfrac 1{m+1} \big( |u|^{m+1} -|g|^{m+1} \big) -\g^m(u-g).
$$
%
When considering exponents $m_i$ instead of $m$, we then write $I_i(u_i,g)$ for $I_{m_i}(u_i,g)$. \\
We define a weak solution to the Cauchy-Dirichlet problem \eqref{equation:global-PME} as follows.

\begin{definition}\label{def:global_weak_solution}
Assume that the vector field $\mathbf A\colon \Omega_T\times \R^N\times\R^{Nn}\to\R^{Nn}$ satisfies \eqref{growth} and \eqref{monotone}. Let $g: \Omega_T \rightarrow \mathbb R^N$ be in the class
\begin{equation*}
	g\in C^0 \big([0,T]; L^{m+1}(\Omega,\R^N)\big)
	\quad\mbox{with}\quad 
	\g^m\in L^2\big(0,T;W^{1,2}(\Omega,\R^N)\big).
\end{equation*}We identify a measurable map
$u\colon\Omega_T\to\R^N$ in the class
\begin{equation*}
	\u^m\in L^2\big(0,T;W^{1,2}(\Omega,\R^N)\big)
\end{equation*} 
with additional assumption $u \in L^{m+1}(\Omega_T,\R^N)$ if $m< 1$, as a \textit{weak solution} to the porous medium type system \eqref{equation:global-PME} with exponent $m$ and initial and boundary values $g$ if and only if $u$ is a weak solution of \eqref{por-med-eq} with exponent $m$ in the sense of Definition \ref{def:weak_solution} and $u$ attains initial and boundary values $g$ in the sense that
\begin{equation} \label{lateral-boundary}
(\u^m - \g^m)(\cdot,t) \in W^{1,2}_0(\Omega),\quad \text{ for a.e. } t \in (0,T),
\end{equation}
and
\begin{align} \label{initial-boundary}
\frac{1}{h} \int_0^h \int_\Omega I(u,g)\, \d x \d t \longrightarrow 0,
\end{align}
as $h \to 0$.
\end{definition}


%

Again, the assumptions on $\varphi$ in Definition~\ref{def:global_weak_solution} can be weakened. It suffices that the test function satisfies
\begin{align*}
\varphi \in W^{1,2}(0,T; L^2(\Omega, \mathbb R^N))
	\cap L^2(0,T;W_0^{1,2}(\Omega, \mathbb R^N))
\end{align*}
and $\varphi(0)=\varphi(T)=0$ when $m \geq 1$.
If $m<1$, we further demand $\partial_t \varphi \in L^{\frac{1+m}{m}}(\Omega_T, \mathbb R^N)$ to ensure the finiteness of the integral of the parabolic part of the equation.

\begin{remark} \label{rem:cont_representative_glob}
As in the local case, we will also prove that a global weak solution $u$ with initial and boundary data $g$ according to Definition~\ref{def:global_weak_solution} has a representative in class $C([0,T]; L^{m+1}(\Omega,\R^N))$. 
\end{remark}

\begin{remark}
Note that for the representative $u \in C([0,T]; L^{m+1}(\Omega,\R^N))$ the condition \eqref{initial-boundary} is equivalent to
\begin{align*} 
u(\cdot, 0 )= g(\cdot,0) \quad \text{ a.e. in } \Omega.
\end{align*}
This is a direct consequence of the estimates in~\eqref{eq:boundary_term_ineq}. 
\end{remark}

For the boundary datum $g \colon \Omega_T\to \R^N$ we suppose that for some $\widetilde m < m$, $\beta > 2\frac m {\tilde m}$ and $\gamma > 1+m$, we have
\begin{align}
  \label{assumption:g}
  \left\{
\begin{aligned}
 \g^{\widetilde m} &\in L^{\beta}\big( 0,T;W^{1,\beta}(\Omega,\R^N)\big),\\
 g&\in C^0\big([0,T],L^{\gamma}(\Omega,\R^N)\big),\\ 
  \partial_t\boldsymbol{g}^{\widetilde m}&\in L^{\frac \gamma {\tilde m}} (\Omega_T,\R^N).
\end{aligned}\right.
\end{align}

The reason for choosing these conditions is twofold. First of all, it ensures that $g$ can be chosen as initial and boundary values for all $i \in \mathbb N$, even though the exponents differ.
Secondly, it ensures the uniform boundedness of the right hand side of the energy estimate in Lemma \ref{lem:caccioppoli_global}. We will show why these conditions are needed in Lemma \ref{lem:g_conditions}.

Observe that we could make stronger but more simplified assumptions, for example $\g^{\widetilde m} \in C^1(\overline{\Omega}_T)$ for some $\widetilde m < m$, which would ensure that the conditions above are satisfied.

Further, note that with these assumptions, the boundary problem for weak solutions might not be well defined for small $i$: The exponent $m_i$, possibly being quite larger than $m$, could exceed the integrability exponent of $g$. However, we are only interested in convergence properties, i.e. the tail of the sequence in question, such that this restriction is of no concern to us. We may thus assume that $m_i$ is already sufficiently close to $m$, ensuring existence of weak solutions and finiteness of the integrals as in Lemma \ref{lem:g_conditions} for all $i \in \mathbb N$.

We will extend the local result in Theorem \ref{theorem} to the boundary:

\begin{theorem}
\label{theorem:global}
Let $(m_i)_{i \in \N}$ be a sequence in $(m_c,\infty)$ such that $m_i \longrightarrow m \in (m_c,\infty)$ as $i \longrightarrow \infty$. Let $u_i$ be a weak solution of Equation~\eqref{equation:global-PME} with exponent $m_i$ in the sense of Definition~\ref{def:global_weak_solution}, where the vector field $\mathbf A$ satisfies the growth and monotonicity conditions~\eqref{growth} and~\eqref{monotone} and the boundary datum $g$ fulfils the conditions \eqref{assumption:g}. \\
Then there exists a subsequence, still denoted by $(\u_i^{m_i})$, and a measurable map $u: \Omega_T \to \R^N$, such that $\u^m \in L^{2}(0,T; W^{1,2}(\Omega,\R^N))$ with
\begin{align} \label{e.global_convergence}
\u_i^{m_i} \stackrel{i \rightarrow \infty}\longrightarrow \u^m \quad \text{ in } L^{2}(0,T; W^{1,2}(\Omega,\R^N)).
\end{align}
Moreover, the limit function $u$ is a weak solution to the Equation~\eqref{equation:global-PME}, attaining the initial and boundary values $g$ in the sense of \eqref{lateral-boundary} and \eqref{initial-boundary}.
\end{theorem}

\begin{remark} \label{rem:uniqueness}
In the special case $\A(x,t,u, D\u^m) = D \u^m$ the limit function $u$ in Theorem~\ref{theorem:global} is unique and the convergence in~\eqref{e.global_convergence} holds for the whole (original) sequence, not only on the level of subsequences.
\end{remark}

\section{Auxiliary results}
\noindent
We first recall a compactness result by Simon~\cite{Simon}. Let us denote $(\tau_h f) (t) := f(t+h)$ for $h>0$.
\begin{lemma} \label{lem:simon-compactness}
	Assume that there is a compact embedding of Banach spaces $X \subset B$. Let $F \subset L^p(0,T;B)$, where $1 \leq p \leq \infty$. In addition, suppose that
	$$
	F \text{ is bounded in } L^1_\mathrm{loc}(0,T;X)
	$$
	and 
	$$
	\| \tau_h f - f\|_{L^p(0,T-h;B)} \to 0 \text{ as } h \to 0, \text{ uniformly for } f \in F. 
	$$
	Then $F$ is relatively compact in $L^p(0,T;B)$.
\end{lemma}

Further there will be the need for some algebraic inequalities, also regarding the boundary term $I(u,g)$. It is often useful to see that it is comparable to $\u ^{\frac{m+1}2} - \g ^{\frac{m+1}2}$. We take the following Lemmas from \cite[Lemma 3.2, 3.3]{Boegelein-Duzaar-Scheven:2018} and from \cite[Lemma 2.3]{Boegelein-Duzaar-Korte-Scheven}.

\begin{lemma}
	For all $\alpha > 1$ there exists $c=c(\alpha)>0$ such that for all $a,b \in \mathbb R^N$ there holds
	\begin{align}\label{eq:pull_exponent_inside}
	|b-a|^\alpha
	\leq c \, |\boldsymbol b^\alpha - \boldsymbol a^\alpha|.
	\end{align}
\end{lemma}

\begin{lemma}
For all $m > 0$ there exists $c=c(m)>0$ such that for all $u,g \in \mathbb R^N$ one has
\begin{align}\label{eq:boundary_term_ineq}
\begin{aligned}
\tfrac 1 c \, I(u,g)
&\leq | \u ^{\frac{m+1}2} - \g ^{\frac{m+1}2} |^2
\leq c \, I(u,g), \\
\tfrac 1 c | \u^m - \g^m |
&\leq \big( |u| + |g| \big)^{m-1}|u-g|
\leq c |\u^m - \g^m|.
\end{aligned}
\end{align}
Further,
\begin{align}\label{eq:boundary_term_ineq2}
\begin{aligned}
I(u,g) &\leq c |\u^m - \g^m|^{(1+m)/m} \quad \text{ for } m > 1, \\
I(u,g) &\leq c (|u| + |g|)^{m-1} |u-g|^2
	\leq c |\u^m - \g^m| |u-g|
	\quad \text{ for } m > 0.
\end{aligned}
\end{align}
\end{lemma}

\subsection{Sobolev-Gagliardo-Nirenberg inequalities}
\hfill \\[1ex]
Next we state the parabolic Sobolev inequality from~\cite[Prop. I.3.1]{DiBe} and a local variant of it. The following inequality will allow us to gain higher integrability for the functions $u_i$ and further, better convergence properties.
\begin{lemma} \label{lem:parabolic-sobolev}
Let $B(x_o,\rho) \subset \Omega$ and $0<t_1<t_2<T$. If 
$$
v \in L^\infty\big(t_1,t_2;L^r(B(x_o,\rho)) \big) \cap L^p\big(t_1,t_2; W^{1,p}(B(x_o,\rho)) \big)
$$
for $p \in (1,\infty)$ and $r \in [1,\infty)$, there exists a constant $c = c(n,p,r)$ such that
\begin{align*}
\int_{t_1}^{t_2} &\int_{B(x_o,\rho)} |v|^\ell \d x \d t \\
&\le c \int_{t_1}^{t_2} \int_{B(x_o,\rho)} \left( \left| \frac{v}{\rho} \right|^p + |D v|^p \right) \d x \d t \left( \sup_{t \in (t_1,t_2)} \int_{B(x_o,\rho) \times \lbrace t \rbrace} |v|^r \d x \right)^\frac{p}{n},
\end{align*}
where $\ell = p \frac{n+r}{n}$.
\end{lemma}

\begin{lemma} \label{lem:global-parabolic-sobolev}
If
$$
v \in L^\infty\big(0,T;L^r(\Omega) \big) \cap L^p\big(0,T; W_0^{1,p}(\Omega) \big)
$$
for $p \in (1,\infty)$ and $r \in [1,\infty)$, there exists a constant $c = c(n,p,r,\Omega)$ such that
\begin{align*}
\iint_{\Omega_T} &|v|^\ell \d x \d t \\
	&\leq c \bigg( \iint_{\Omega_T} |Dv|^p \d x \d t \bigg) 
	\bigg( \sup_{t \in (0,T)} \int_{\Omega \times \lbrace t \rbrace} |v|^r \d x \bigg)^{\frac p n},
\end{align*}
where $\ell = p \frac{n+r}{n}$.
\end{lemma}

%
%
%

\subsection{Mollification in time}
\hfill \\[1ex]
In order to be able to prove useful estimates for weak solutions of Equation~\eqref{por-med-eq}, we exploit time mollification of the following type. 

\begin{definition} For $v \in L^1(\Omega_T,\mathbb{R}^N)$ and $h>0$, define a mollification in time by
\begin{align*}
\mollifytime{v}{h}(x,t)
	:= \frac 1 h \int_0^t e^{\frac{s-t}{h}} v(x,s) \mathrm{d}s.
\end{align*}
Similarly, define the reverse time mollification in time by
\begin{align*}
\mollifytime{v}{\bar h}(x,t)
	:= \frac 1 h \int_t^T e^{\frac{t-s}{h}} v(x,s) \mathrm{d}s.
\end{align*}
\end{definition}
We collect some useful properties of the mollification in the following Lemma, see~\cite[Lemma 2.9]{Kinnunen-Lindqvist} and~\cite[Appendix B]{BDM:pq}. Analogous statements hold true for the reverse time mollification.

\begin{lemma} \label{lem:mollifier}
  Let $v$ and $\mollifytime{v}{h}$ be as above. Then the following properties hold: \\
  \emph{(i)} If $v \in L^p(\Omega_T,\R^N)$ for some $p \geq 1$, then 

  \begin{equation*}
    \| \mollifytime{v}{h} \|_{L^p(\Omega_T,\R^N)} \leq \| v \|_{L^p(\Omega_T,\R^N)},
  \end{equation*}
  and $\mollifytime{v}{h} \to v$ in $L^p(\Omega_T,\R^N)$ as $h \to 0$.
\\
\emph{(ii)} Let $v\in L^p(0,T; W^{1,p}(\Omega,\R^N))$ for some $p \geq 1$. Then
\begin{equation*}
  \| \mollifytime{v}{h}\|_{L^p(0,T;W^{1,p}(\Omega,\R^N))} \leq \| v\|_{L^p(0,T;W^{1,p}(\Omega,\R^N))} 
\end{equation*}
and $\mollifytime{v}{h}\to v$ in $L^p(0,T;W^{1,p}(\Omega,\R^N))$ as $h \to 0$.
\\
\emph{(iii)} If $v \in L^p(0,T; W_0^{1,p}(\Omega,\R^N))$, then $\mollifytime{v}{h} \in L^p(0,T;W_0^{1,p}(\Omega,\R^N))$.
\vspace{1mm}
\\
\emph{(iv)} If $v \in L^p(0,T; L^p(\Omega,\R^N))$, then $\mollifytime{v}{h} \in C([0,T]; L^p(\Omega,\R^N))$.
\vspace{1mm}
\\
\emph{(v)} The weak time derivative $\partial_t \mollifytime{v}{h}$ exists in $\Omega_T$ and is given by formula
\begin{equation*}
  \partial_t \mollifytime{v}{h} = \frac{1}{h} ( v - \mollifytime{v}{h} ),
\end{equation*}
whereas for the reverse mollification we have 
$$
\partial_t \mollifytime{v}{\bar h} = \frac{1}{h} (\mollifytime{v}{\bar h} - v).
$$
\end{lemma}

\begin{remark}
Observe that similar properties hold also for mollification defined as
$$
\mollifytime{v}{h}(x,t)
	:= e^{-\frac{t}{h}} v_o  + \frac 1 h \int_0^t e^{\frac{s-t}{h}} v(x,s) \mathrm{d}s
$$
for $v_o \in L^1(\Omega,\R^N)$. One advantage of this formula is that we can compute
$$
\partial_t \mollifytime{v}{h}(x,t) = \frac 1 h \int_0^t e^{\frac{s-t}{h}} \partial_s v(x,s) \mathrm{d}s
$$
under suitable assumptions, see~\cite[Appendix B, Lemma B.3]{BDM:pq}. From this together with Lemma~\ref{lem:mollifier} one can deduce convergences for the time derivative as well provided that it exists in an appropriate space. We will exploit this in the global case.
\end{remark}

\section{Continuity in time and mollified formulation}
In this section we will prove that weak solutions have representatives that are continuous in time, according to Remarks~\ref{rem:cont_representative_loc} and~\ref{rem:cont_representative_glob}.

\subsection{Continuity in time for local problem} \label{section:loc_continuity}
\hfill \\[1ex]
In order to prove the continuity in time of weak solution according to Definition~\ref{def:weak_solution} as noted in Remark~\ref{rem:cont_representative_loc}, we will use the following Lemma, which can be found in \cite[Lemma 2.12]{Sturm}, \cite[Lemma 3.4]{Singer-Vestberg}, \cite[Lemma 3.8]{Vespri-Vestberg}. We include the proof for the continuity for completeness, where we take the approach of \cite{Vespri-Vestberg}. Observe that in the local case we will use mollifications for $\u^m$ defined by 
\begin{align*}
\mollifytime{\u^m}{h}(x,t)
	:= \frac 1 h \int_{\tau_1}^t e^{\frac{s-t}{h}} \u^m(x,s) \mathrm{d}s,
\end{align*}
for $t \geq \tau_1 $, in which $\tau_1 > 0$ is fixed. This is due to the fact that $\u^m$ is only locally integrable. Similarly, define the reverse time mollification in time by
\begin{align*}
\mollifytime{\u^m}{\bar h}(x,t)
	:= \frac 1 h \int_t^{\tau_2} e^{\frac{t-s}{h}} \u^m(x,s) \mathrm{d}s,
\end{align*}
for $t \leq \tau_2 < T$.

\begin{lemma}\label{lem:continuity_local_lem}
Let $\mathcal V$ be the set of all $v \in C^0 \big( (0,T), L^{1+m}_{\loc}(\Omega, \R^N) \big)$ such that
\begin{align*}
\v^m \in L^2_{\loc} \big( 0,T; W^{1,2}_{\loc}(\Omega, \R^N) \big)
\quad\text{ and }\quad
\partial_t \v^m \in L^\frac{m+1}{m}_{\loc}(\Omega_T, \R^N).
\end{align*}
Then, for a weak solution $u$ according to Definition~\ref{def:weak_solution},
\begin{align*}
 \iint_{\Omega_T} \partial_t \zeta I(u,v) \dx \dt
&= \iint_{\Omega_T} \zeta \partial_t \v^m \cdot (u-v)
+  \A(x,t,u,D\u^m) \cdot D \big( \zeta (\u^m - \v^m) \big) \dx \dt
\end{align*}
holds true for all $v \in \mathcal V,
\; \zeta \in C_0^\infty(\Omega_T)$.
\end{lemma}


\begin{proof}
We will test the weak equation \eqref{weak-solution} for $u$ with $\varphi = \zeta( \v^m - \mollifytime {\u^m} h )$ with some small fixed $\tau_1 > 0$ in the mollifier.
We first inspect the parabolic part of the equation:
\begin{align*}
\iint_{\Omega_T} u \cdot \partial_t \varphi \dx \dt
&= \iint_{\Omega_T} \partial_t \zeta u \cdot
( \v^m - \mollifytime {\u^m} h )
+ \zeta u
\cdot \partial_t  (\v^m - \mollifytime {\u^m} h )
\dx \dt \\
&=\iint_{\Omega_T} \partial_t \zeta u \cdot
( \v^m - \mollifytime {\u^m} h )
+ \zeta u  \cdot \partial_t  \v^m \dx \dt \\
&\quad +\iint_{\Omega_T}- \zeta \mollifytime{\u^m}{h}^{1/m}
\cdot \partial_t \mollifytime{\u^m}{h}
+ \zeta ( \mollifytime{\u^m}{h}^{1/m} - u )
\cdot \partial_t \mollifytime {\u^m} h )
\dx \dt \\
&\leq \iint_{\Omega_T} \partial_t \zeta u \cdot
( \v^m - \mollifytime {\u^m} h )
+ \zeta u  \cdot \partial_t  \v^m \dx \dt \\
&\quad +\frac{m}{m+1}\iint_{\Omega_T} \partial_t \zeta | \mollifytime{\u^m} h |^{(m+1)/m}
\dx \dt \\
&\stackrel{h \downarrow 0}\longrightarrow
\iint_{\Omega_T} \partial_t \zeta ( u \cdot \v^m - |u|^{m+1})
+ \zeta u \cdot \partial_t \v^m
+ \frac{m}{m+1} \partial_t \zeta |u|^{m+1} \dx \dt \\
&= \int_{\Omega_T} \zeta \partial_t \v^m \cdot (u-v)
- \partial_t \zeta I(u,v) \dx \dt.
\end{align*}
By comparing to the divergence part of the equation, we obtain the direction '$\leq$' of the claim. The other direction can be shown by taking the reverse time mollification in the test function $\varphi$.
\end{proof}

\begin{lemma}\label{lem:continuity_local}
The weak solution $u$ according to Definition~\ref{def:weak_solution} has a representative that belongs to class $C^0((0,T);L_{\loc}^{1+m}(\Omega, \R^N))$. 
\end{lemma}

\begin{proof}
Let $K \Subset \Omega$ and $\eta \in C_0^\infty(\Omega)$ with $\eta = 1$ on $K$, $|D\eta| \leq C(K)$. Let $\tau \in (0, \tfrac 1 2 T)$ and $\varepsilon>0$ such that $\tau+\varepsilon < \tfrac 1 2 T$. Take $\psi \in C^\infty([0,\infty])$ with $\psi =1$ on $[0,\tfrac 1 2 T]$, $\psi = 0$ on $[\tfrac 3 4 T, T]$ and $|\psi'| \leq 8T^{-1}$.
Further define
\begin{align*}
\xi(t) = \xi_{\varepsilon,\tau}(t) := \begin{cases}
0, & t < \tau, \\
\tfrac 1 \varepsilon (t - \tau), & t \in [\tau, \tau + \varepsilon], \\
1, & t > \tau + \varepsilon.
\end{cases}
\end{align*}
\newcommand{\w}{\boldsymbol{w}}
We apply Lemma \ref{lem:continuity_local_lem} with $\zeta = \eta \psi \xi$, $w_h = \mollifytime{\u^m}{\bar h}^{1/m}$, and $\tau_2 \in (\frac 3 4 T,T)$ in the mollifier. Observe that $w_h \in \mathcal  V$ defined in Lemma~\ref{lem:continuity_local_lem}. Especially $w_h \in C^0([t_1,\frac 3 4 T];L_{\loc}^{1+m}(\Omega,\R^N))$ for any $h>0$ and $t_1 > 0$ can be seen as follows. If $m > 1$ this is a direct consequence of~\eqref{eq:pull_exponent_inside} and the fact that $\mollifytime{\u^m}{\bar h} \in C^0([t_1,\frac 3 4 T];L_{\loc}^\frac{1+m}{m}(\Omega,\R^N))$ together with assumption $u \in L^{m+1}_{\loc}(\Omega_T,\R^N)$ and Lemma~\ref{lem:mollifier} (iv). In the case $m < 1$, we make use of second inequality in~\eqref{eq:boundary_term_ineq} and H\"older's inequality. Fixing $0< t_1 \leq \tau$ and defining $E := \mathrm{supp}\, \eta \times (t_1,\frac 3 4 T)$, this yields
\begin{align*}
\frac 1 \varepsilon &\iint_{\Omega \times (\tau,\tau+\varepsilon)}
	\eta I(u,w_h) \dx \dt
=
\iint_{\Omega_T} \eta \partial_t \xi I(u,w_h) \dx \dt \\
	&= \iint_{\Omega_T} \left( \zeta \partial_t \w_h^m (u-w_h) 
		+ \A(x,t,u,D\u^m) \cdot D\big( \zeta (\u^m - \w_h^m) \big) 
		-I(u,w_h)  \eta \xi \partial_t \psi \right)
		\dx \dt \\
	&\leq  C \iint_{E} \left( |D\u^m| \big(
		|D\u^m - D\w_h^m| + |D\eta||\u^m - \w_h^m|
	\big)
		+ \frac 8 T I(u,w_h) \right) \dx \dt \\
	&\leq C \iint_{E} \left( |D\u^m - D\w_h^m|^2
		+|\u^m - \boldsymbol{w_h}^m|^2
		+ I(u,w_h) \right) \dx \dt
\end{align*}
by applying H\"older's inequality and using the assumption $|D\u^m| \in L^2_{\loc}(\Omega_T)$. By the convergence properties of the time mollification, the first and second term on the right hand side vanish as $h \downarrow 0$. \\
For the last term, we must distinguish between two cases. If $m \geq 1$, we can apply estimate $I(u,w_h) \leq c |\u^m-\w_h^m|^{(1+m)/m}$ by \eqref{eq:boundary_term_ineq2}, which vanishes as $h \downarrow 0$. In the singular case $m <1$, by \eqref{eq:boundary_term_ineq2} we compute
\begin{align*}
&\iint_{E} I(u,w_n) \dx \dt \\
	&\quad\leq c \iint_{E} 
		|\u^m-\w_h^m| |u-w_h|
	\dx \dt \\
	&\quad\leq c \bigg( 
		\iint_{E}
		|\u^m-\w_h^m|^\frac{1+m}{m}
	\dx \dt
	\bigg)^\frac{m}{1+m}
	\bigg(
		\iint_{E}
		|u-w_h|^{1+m}
	\dx \dt
	\bigg)^\frac{1}{1+m}.
\end{align*}
Again, since $\u^m \in L^\frac{1+m}{m}_{\loc}(\Omega_T,\R^N)$ and $u \in L_{\loc}^{1+m}(\Omega_T,\R^N)$, these integrals vanish as $h \downarrow 0$. Notice that in the previous estimates, the right hand side does not depend on $\tau$. Thus,
\begin{align}\label{eq:cont_bdry_term_vanishes}
\lim_{h \downarrow 0}
\sup_{\tau \in (t_1,\frac 1 2 T)}
\int_{K \times \lbrace \tau \rbrace} I(u,w_h) \dx
= 0.
\end{align}
For the case $m \geq 1$, we use the inequalities \eqref{eq:boundary_term_ineq} and \eqref{eq:pull_exponent_inside} and see that
\begin{align*}
\sup_{\tau \in (t_1,\frac 1 2 T)}
	\int_{K \times \lbrace \tau \rbrace } |w_h - u|^{m+1} \dx
	&\leq C\sup_{\tau \in (t_1,\frac 1 2 T)} \int_{K \times \lbrace \tau \rbrace } I(u,w_h) \dx \rightarrow 0.
\end{align*}
In the case $m < 1$ we use inequalities~\eqref{eq:boundary_term_ineq} and~\eqref{eq:boundary_term_ineq2} to obtain
\begin{align*}
\int_{K \times \lbrace \tau \rbrace } &|w_h - u|^{m+1} \dx \\
&\leq  c\int_{K \times \lbrace \tau \rbrace } \left( |w_h| + |u| \right)^\frac{(1-m)(1+m)}{2} |\w_h^\frac{m+1}{2} - \u^\frac{m+1}{2}|^{m+1} \dx \\
&\leq c \left( \int_{K \times \lbrace \tau \rbrace } \left( |w_h| + |u| \right)^{m+1}\d x \right)^\frac{1-m}{2} \left(  \int_{K \times \lbrace \tau \rbrace }|\w_h^\frac{m+1}{2} - \u^\frac{m+1}{2}|^2 \dx \right)^\frac{m+1}{2} \\
&\leq c\left( \int_{K \times \lbrace \tau \rbrace } |w_h|^{m+1} \d x + \int_{K \times \lbrace \tau \rbrace } |u|^{m+1} \d x \right)^\frac{1-m}{2} \left(  \int_{K \times \lbrace \tau \rbrace }I(u, w_h) \dx \right)^\frac{m+1}{2}
\end{align*}
for any $\tau \in (t_1, \frac 1 2 T)$. By taking supremum over $\tau$ and passing to the limit $h \to 0$ the right hand side converges to zero by \eqref{eq:cont_bdry_term_vanishes}. Observe that the first term of the right hand side stays bounded since we have that $u \in L^\infty_{\loc} (0,T; L^{m+1}_{\loc} (\Omega,\R^N))$. This is true by Lemma~\ref{lem:caccioppoli} and since $u \in L_{\loc}^{m+1}(\Omega, \R^N)$ by definition. In addition, from the properties of the mollification it follows that
$$
\left\| \mollifytime{\u^m}{\bar h}^{1/m}  \right\|_{L^\infty (t_1,\frac 3 4 T; L^{m+1}(K, \R^N))} \leq \left\| u \right\|_{L^\infty (t_1, \frac 3 4 T; L^{m+1} (K, \R^N))},
$$
which implies that $w_h$ is uniformly bounded in $L^\infty (t_1,\frac 3 4 T; L^{m+1}_{\loc} (\Omega, \R^N))$.
We can thus come to the same conclusion as in the case $m \geq 1$.
In total, we have shown that
\begin{align*}
w_h = \mollifytime{\u^m}{\bar h}^{1/m} \longrightarrow u \quad\text{ in }
L^\infty(t_1,\tfrac 1 2 T;L^{1+m}(K, \R^N))
\end{align*}
as $h \downarrow 0$. Observe that $K \Subset \Omega$ and $t_1 > 0$ were arbitrary. Since $w_h$ is continuous map from $(t_1,t_2)$ to $L_{\loc}^{m+1}(\Omega,\R^N)$, we conclude that $u \in C^0((t_1,\frac 1 2 T ],L^{1+m}_{\loc}(\Omega, \R^N))$ as uniform limit of continuous functions from $(t_1,t_2)$ to $L_{\loc}^{m+1}(\Omega,\R^N)$, after possible redefinition in a set of measure zero. To obtain the result on the full time interval, one can either modify cut-off functions $\xi$ and $\psi$ so that $\tau$ can be arbitrarily close to $T$ on a compact subinterval of $(0,T)$, or apply the same arguments with usual time mollifications $w_h = \mollifytime{\u^m}{h}^{1/m}$
and reversed cut-off functions, as suggested in \cite[Lemma 3.9]{Vespri-Vestberg}. This completes the proof.
\end{proof}

\subsection{Continuity in time in global case}
%

\begin{lemma}
The weak solution $u$ according to Definition~\ref{def:global_weak_solution} has a representative that belongs to class $C^0([0,T];L^{1+m}(\Omega, \R^N))$.
\end{lemma}

\newcommand{\mtul}{\mollifytime{\u^m}{\bar \lambda}}
\newcommand{\mtgl}{\mollifytime{\g^m}{\bar \lambda}}
\newcommand{\mtuh}{\mollifytime{\u^m}{h}}
\newcommand{\mtgh}{\mollifytime{\g^m}{h}}

\begin{proof}
Let $\tau \in (0,\tfrac 1 2 T)$, $\varepsilon>0$ such that $\tau + \varepsilon < \tfrac 1 2 T$.
Define $\zeta=\psi \xi$ with $\psi, \xi$ as in Lemma~\ref{lem:continuity_local}. We test the weak formulation \eqref{weak-solution} for $u$ against the test function $\varphi = \zeta ( \mtul - \mtuh - \mtgl + \mtgh )$ for two different mollification parameters $\lambda >0, h>0$. Now the mollifications are defined as 
\begin{align*}
	\mollifytime{\u^m}{h}(x,t)
	&:= e^{-\frac{t}{h}}\g^m(x,0) +  \frac 1 h \int_{0}^t e^{\frac{s-t}{h}} \u^m(x,s) \mathrm{d}s, \\
\mollifytime{\g^m}{h}(x,t)
	&:= e^{-\frac{t}{h}}\g^m(x,0) + \frac 1 h \int_{0}^t e^{\frac{s-t}{h}} \g^m(x,s) \mathrm{d}s, 
\end{align*}
and
\begin{align*}
\mollifytime{\u^m}{\bar \lambda}(x,t)
	&:= e^\frac{t- T}{\lambda}\g^m(x,T) + \frac 1 \lambda \int_{t}^{T} e^\frac{t - s}{\lambda} \u^m(x,s) \mathrm{d}s, \\
	\mollifytime{\g^m}{\bar \lambda}(x,t)
	&:= e^\frac{t- T}{\lambda}\g^m(x,T) + \frac 1 \lambda \int_{t}^{T} e^\frac{t - s}{\lambda} \g^m(x,s) \mathrm{d}s.
\end{align*}
Notice that $\varphi(\cdot,t) \in W_0^{1,2}(\Omega,\R^N)$ for a.e. $t \in (0,T)$. 
For the parabolic part of the equation we have
\begin{align*}
\iint_{\Omega_T} u \cdot \partial_t \varphi \dx \dt
	&=
	\iint_{\Omega_T} \partial_t \zeta u \cdot
		( \mtul - \mtuh ) + \zeta u \cdot \partial_t ( \mtul - \mtuh ) \dx \dt \\
	&\quad - \iint_{\Omega_T} \partial_t \zeta u \cdot
	( \mtgl - \mtgh ) + \zeta u \cdot \partial_t ( \mtgl - \mtgh ) \dx \dt.
\end{align*}
The first line can be estimated as in Lemma~\ref{lem:continuity_local_lem}, while in the second line one can immediately pass to the limit $h \downarrow 0$. Thus we obtain
\begin{align*}
\iint_{\Omega_T} &\partial_t \xi I(u, \mtul^{1/m}) \dx \dt \\
	&\leq \iint_{\Omega_T} \zeta \partial_t \mtul
		\cdot ( u - \mtul^{1/m})
		+ \zeta \A(x,t,u,D\u^m) \cdot D \big( \u^m- \mtul ) \dx \dt \\
	&\quad + \iint_{\Omega_T} \partial_t \zeta u \cdot
		( \g^m - \mtgl )
		+ \zeta u \cdot \partial_t ( \g^m - \mtgl ) \dx \dt \\
	&\quad - \iint_{\Omega_T} \zeta \A(x,t,u,D\u^m)
		\cdot D \big( \g^m- \mtgl ) \dx \dt \\
		&\quad - \iint_{\Omega_T}\partial_t \psi I(u, \mtul^{1/m}) \dx \dt \\
	&=: \mathrm{I} +\mathrm{II} +\mathrm{III} + \mathrm{IV}.
\end{align*}
Observe that the term on the left hand side is non-negative.  The first integral I together with IV will vanish due to the same reasoning as in Lemma \ref{lem:continuity_local}. The third integral III vanishes as $\lambda \downarrow 0$ due to the growth condition \eqref{growth}, the assumption $|D\u^m| \in L^2(\Omega_T)$ and Lemma~\ref{lem:mollifier} (ii). For II, we estimate
\begin{align*}
\mathrm{II}
	&\leq  \bigg( \frac 1 \varepsilon \iint_{\Omega \times (\tau, \tau+\varepsilon)} |u|^{m+1} \dx \dt \bigg)^\frac{1}{1+m}
		\bigg( \frac 1 \varepsilon \iint_{\Omega \times (\tau, \tau+\varepsilon)} | \mtgl - \g^m |^{\frac{1+m}{m}} \d x \d t
		\bigg)^\frac{m}{1+m} \\
	&\quad + \frac{8}{T} \bigg( \iint_{\Omega_T}
	|u|^{m+1} \dx \dt \bigg)^\frac{1}{m+1}
	\bigg( \iint_{\Omega_T} | \mtgl - \g^m |^{\frac{1+m}{m}}
	\bigg)^\frac{m}{1+m} \\
	&\quad + \bigg( \iint_{\Omega_T} |u|^{m+1} \dx \dt
	\bigg)^\frac{1}{m+1}
	\bigg( \iint_{\Omega_T} 
	|\partial_t \mtgl - \partial_t \g^m | ^{\frac{m+1}{m}} 
	\bigg)^\frac{m}{1+m}
\end{align*}
After passing to the limit $\eps \searrow 0$ the first term equals
\begin{align*}
\bigg(\int_{\Omega \times \{ \tau \}} |u|^{m+1} \dx \bigg)^\frac{1}{1+m}
		\bigg(\int_{\Omega \times \{ \tau\}} | \mtgl - \g^m |^{\frac{1+m}{m}} \d x 		\bigg)^\frac{m}{1+m}
\end{align*}
for a.e. $\tau \in (0,\frac 1 2 T)$. By Caccioppoli inequality in Lemma~\ref{lem:caccioppoli_global} (see also Remark~\ref{rem:caccioppoli_glob}) the assumption $u \in L^{1+m}(\Omega_T)$ and assumptions for $g$ the first integral is uniformly bounded in $\tau$, and the second integral vanishes as $\lambda \downarrow 0$ by properties of mollification. 
%
%
%
Using $u \in L^{1+m}(\Omega_T,\R^N)$, as well as $\g^m, \partial_t \g^m \in L^{(1+m)/m}(\Omega_T,\R^N)$, it follows that II vanishes as $\lambda \downarrow 0$. For the remaining terms, we can use the same arguments as in Lemma \ref{lem:continuity_local} to obtain
\begin{align*}
\lim_{\lambda \downarrow 0} \sup_{\tau \in (0,\frac 1 2 T)} \int_{\Omega} |u-\mtul^{1/m}|^{m+1} \dx
	 = 0.
\end{align*}
This proves that $u$ is a uniform limit of functions in $C^0([0,\frac 1 2 T];L^{1+m}(\Omega,\R^N))$, which implies the existence of continuous representative $u \in C^0([0,\frac 1 2 T]; L^{m+1}(\Omega,\R^N))$. For the interval $[\frac 1 2 T,T]$, we may use reversed mollifications and cut-off functions as mentioned at the end of the proof of Lemma~\ref{lem:continuity_local}. Finally, we obtain that $u \in C^0([0,T];L^{1+m}(\Omega,\R^N))$.
\end{proof}

\subsection{Mollified formulation}

In order to prove useful estimates we use the following mollified formulation of~\eqref{weak-solution}, which can be derived similarly as in~\cite{Boegelein-Duzaar-Gianazza}. We include the proof for expository purposes. In the local case we use mollifications
\begin{align*}
\mollifytime{u}{h}(x,t)
	:= \frac 1 h \int_{\tau_1}^t e^{\frac{s-t}{h}} u(x,s) \mathrm{d}s,
\end{align*}
and 
\begin{align*}
\mollifytime{\varphi}{\bar h}(x,t)
	:= \frac 1 h \int_{t}^T e^{\frac{t-s}{h}} \varphi(x,s) \mathrm{d}s,
\end{align*}
for solution $u$ and test function $\varphi$.

\begin{lemma}
For a local weak solution $u$ according to Definition~\ref{def:weak_solution}, one has
\begin{align}\label{eq:mollifiedEquation_loc}
\int_{\tau_1}^T \int_{\Omega} \big(
	\partial_t \mollifytime u h \cdot \varphi 
	+ \mollifytime{A(x,t,u,D\u^m)}h \cdot D \varphi
\big) \d x \d t 
	= \frac{1}{h} \int_\Omega u(\tau_1) \cdot \int_{\tau_1}^T e^{\frac {\tau_1- s}{ h}} \varphi\, \mathrm{d}s \d x
\end{align}
for all $\varphi \in L^2(0,T;W^{1,2}(\Omega,\mathbb{R}^N))$ with the property $\mathrm{supp }\,  \varphi \Subset \Omega_T $, and almost every $\tau_1 \in (0,T)$. In the case $m < \frac{n-1}{n+1}$, we suppose in addition that $\varphi \in L^\frac{1+m}{m}(\Omega_T, \mathbb R^N)$. If we take the representative $u \in C^0((0,T);L^{m+1}_{\loc}(\Omega,\R^N))$, the formulation holds for every $\tau_1 \in (0,T)$. 
\end{lemma}

\begin{proof} For $\varepsilon, \tau >0$ define $\eta_\varepsilon \in W^{1,\infty}([0,T],[0,1])$ with $\eta_\varepsilon(t) := \frac{t-\tau_1}{\varepsilon}$ on $[\tau_1,\tau_1 + \varepsilon]$, $\eta_\varepsilon(t) := 1$ on $(\tau_1 + \varepsilon,T]$ and $\eta_\varepsilon (t) := 0$ on $[0,\tau_1]$. We insert $\mollifytime \varphi {\bar h} \eta_\varepsilon$ as test function into \eqref{weak-solution}. Note that $\mollifytime \varphi {\bar h}(\cdot,t) = 0$ if $t$ close to $T$ since $\varphi$ is compactly supported, while the cutoff function $\eta_\varepsilon$ takes care of the initial values. It follows that
\begin{align*}
0
	&=
	\iint_{\Omega_T}\big[
	- u\cdot\mollifytime \varphi{\bar h} \partial_t\eta_\varepsilon
	- u\cdot\partial_t\mollifytime \varphi{\bar h} \eta_\varepsilon
	+ \mathbf \eta_\varepsilon \A(x,t,u,D\u^m)\cdot D\mollifytime \varphi{\bar h}\big]\dx\dt.
\end{align*}
We inspect the first integral and obtain
\begin{align*}
&-\iint_{\Omega_T}
	 u\cdot\mollifytime \varphi{\bar h} \partial_t\eta_\varepsilon \dx\dt
	= -\frac 1 \varepsilon \iint_{\Omega \times [\tau_1,\tau_1+ \varepsilon]} u\cdot\mollifytime \varphi{\bar h}  \dx \dt \\
	\stackrel{\varepsilon \downarrow 0} \longrightarrow &- \int_\Omega u(\tau_1)\cdot\mollifytime \varphi{\bar h}(t_1)  \dx
	= -\frac 1 h \int_\Omega u(x,\tau_1) \cdot \int_{\tau_1}^T e^{-\frac s h} \varphi(x,s) \ds  \dx.
\end{align*}
After passing to the limit $\varepsilon \to 0$, we use Fubini's theorem and calculate
\begin{align*}
\iint_{\Omega_T} u \cdot \mollifytime \varphi{\bar h}  \eta \dx \dt
	&= \frac 1 h \int_\Omega \int_{\tau_1}^T \int_t^T u(x,t) \cdot e^{\frac{t-s}{h}} \varphi(x,s)  \ds \dt \dx \\
	&= \frac 1 h \int_\Omega \int_{\tau_1}^T \int_{\tau_1}^s u(x,t) \cdot e^{\frac{t-s}{h}} \varphi(x,s)  \dt \ds \dx \\
	&= \int_\Omega \int_{\tau_1}^T  \varphi(x,t) \cdot \frac 1 h\int_{\tau_1}^t e^{\frac{s-t}{h}} u(x,s)   \ds  \dt \dx \\
	&= \iint_{\Omega_T} \eta \mollifytime u h \cdot \varphi \dx \dt,
\end{align*}
where we renamed $s,t$ by $t,s$ in the second last equality. Here $\eta (t) := \lim_{\varepsilon \to 0} \eta_\varepsilon (t) =  \chi_{(t_1,T]} (t)$. The same can be done to transfer the reverse time mollification from $D\varphi$ onto the vector field $\A$ by the property $D( \mollifytime \varphi {\bar h}) = \mollifytime {D\varphi}{\bar h}$, which transforms it again into the forward time mollification. This deals with the divergence part of the equation. \\
For the remaining parabolic part we use the previous equality and the properties
\begin{align*}
\partial_t \mollifytime \varphi h = \tfrac 1 h ( \varphi - \mollifytime \varphi h),
\quad
\partial_t \mollifytime \varphi {\bar h} = -\tfrac 1 h ( \varphi - \mollifytime \varphi {\bar h})
\end{align*}
to obtain the following:
\begin{align*}
-\iint_{\Omega_T} &u \cdot \partial_t \mollifytime \varphi{\bar h}  \eta_\varepsilon \dx \dt \\
	&= \frac 1 h \iint_{\Omega_T} u \cdot (\varphi - \mollifytime \varphi{\bar h})  \eta \dx \dt
	= \frac 1 h \iint_{\Omega_T} (u- \mollifytime u h) \cdot \varphi  \eta  \dx \dt \\
	&= \iint_{\Omega_T} \partial_t \mollifytime u h \cdot \varphi  \eta \dx \dt.
\end{align*}
By collecting these results we obtain the equation~\eqref{eq:mollifiedEquation_loc}. The right hand side of the equation converges to zero by dominated convergence theorem.
\end{proof}
In the same manner in the global case we have the following formulation.

\begin{lemma}
For a global weak solution $u$ according to Definition~\ref{def:global_weak_solution} in class $C^0([0,T];L^{m+1}(\Omega,\R^N))$, one has
\begin{align}\label{eq:mollifiedEquation}
\iint_{\Omega_T} \big(
	\partial_t \mollifytime u h \cdot \varphi 
	+ \mollifytime{A(x,t,u,D\u^m)}h \cdot D \varphi
\big) \d x \d t 
	= \frac{1}{h} \int_\Omega u(0) \cdot \int_{0}^T e^{-\frac s h} \varphi \mathrm{d}s\dx
\end{align}
for all $\varphi \in L^2(0,T;W_0^{1,2}(\Omega,\mathbb{R}^N))$. In the case $m < \frac{n-1}{n+1}$, we suppose in addition that $\varphi \in L^\frac{1+m}{m}(\Omega_T, \mathbb R^N)$.
\end{lemma}

\begin{remark}
Regarding the finiteness of the integrals in \eqref{eq:mollifiedEquation}, the degenerate case is clear, since $u \in L^{2m}(\Omega_T,\mathbb R^N)$ and $\A$ fulfils \eqref{growth}. The singular case $m<1$ requires further inspection.

We can use the Gagliardo-Nirenberg inequality for $v=\u^m$ in Lemma \ref{lem:parabolic-sobolev} to see that $u \in L^{\gamma}(\Omega_T,\mathbb R^N)$, where $\gamma = {2(nm+m+1)}/{n}$. If $m \geq \frac{n-1}{n+1}$, then $\gamma \geq 2$, which is sufficient for the finiteness of the first integral in \eqref{eq:mollifiedEquation}. In the remaining case we have $\partial_t \mollifytime u h \in L^{1+m}(\Omega_T,\mathbb R^N)$ by applying the previous Lemma \ref{lem:mollifier} and the fact $u \in L^{1+m}(\Omega_T,\mathbb R^N)$. Combining this with the additional assumption $\varphi \in L^{{(1+m)}/m}(\Omega_T,\mathbb R^N)$ shows the finiteness in this case.
\end{remark}

\section{Stability in the local setting}

\subsection{Energy estimate}
\hfill \\[1ex]
In this section we prove an energy estimate by using the mollified formulation~\eqref{eq:mollifiedEquation}. This is an essential tool to conclude boundedness and weak convergence for the gradients of a sequence of weak solutions. For the proof, we proceed similar to \cite{Boegelein-Lukkari-Scheven}.

\begin{lemma} \label{lem:caccioppoli}
Let $u$ be a weak solution to Equation~\eqref{por-med-eq} in the sense of Definition~\ref{def:weak_solution}, where the vector field $\mathbf A$ satisfies growth conditions~\eqref{growth}. Let further be $0<\delta<t_1<t_2<T$ and $\eta \in C_0^\infty(\Omega)$. Then we have the estimate
\begin{align*}
	& \sup_{t \in (t_1,t_2)} 
 \int_{\Omega \times \lbrace t \rbrace} 
	\eta^2 |u|^{m+1} \dx +
	\iint_{\Omega \times (t_1,t_2) } \eta^2 \big|D\u^m\big|^2 \dx\dt \\
	&\quad\leq 
	\frac{c_m}{\delta}\, \iint_{\Omega\times (t_1-\delta,t_1) } \eta^2 |u|^{m+1} \dx\dt +
	c_m\, \iint_{\Omega\times (t_1-\delta,t_2) }   |D\eta|^2 |\u^m|^{2} \dx\dt
\end{align*}
where $c_m=c(m,\nu,L)>0$ with
\begin{align*}
\sup_{m \in (m_c,M)} c_m < \infty
\end{align*}
for any given $M>m_c$.
\end{lemma}


\begin{proof}
For $\tau \in (t_1,t_2)$ and $0<\sigma<\tau$ we choose $\alpha(t)=\alpha_\sigma(t)$ and $\xi(t)$ as
\begin{align*}
\alpha(t) = \begin{cases}
1, & t \in [0,\tau-\sigma), \\
\tfrac {\tau-t} \sigma,  & t \in [\tau - \sigma, \tau), \\
0, & t \in [\tau,T]
\end{cases} \text{ and }
\xi(t) = \begin{cases}
0, & t \in [0,t_1-\delta), \\
\tfrac{t-(t_1-\delta)}{\delta}, & t \in [t_1-\delta,t_1), \\
1, & t \in [t_1,T].
\end{cases}
\end{align*}
We apply the test function $\varphi = \alpha \xi \eta^2 \u^m$ to the mollified equation \eqref{eq:mollifiedEquation}, with $\tau_1 \in (0,t_1-\delta)$. By adding and subtracting $ \alpha \xi \eta^2 \partial_t \mollifytime u h \cdot \mollifytime {\u} h^m$ and using the formula $\partial_t \mollifytime u h = \frac{u- \mollifytime u h}{h}$ from Lemma \ref{lem:mollifier}, we obtain
\begin{align*}
\iint_{\Omega_T} & \alpha \xi \eta^2 \partial_t \mollifytime u h \cdot  \u^m \dx \dt \\
	&= \iint_{\Omega_T} \alpha \xi \eta^2 \partial_t \mollifytime u h \cdot \mollifytime {\u} h^m \dx \dt
		+ \frac 1 h \iint_{\Omega_T} \alpha \xi \eta^2 (u- \mollifytime u h) \cdot (\u^m-\mollifytime \u h^m) \dx \dt \\
	&\geq  \iint_{\Omega_T} \alpha \xi \eta^2 \partial_t \mollifytime u h \cdot  \mollifytime {\u} h^m \dx \dt,
\end{align*}
as $(a-b) \cdot (\mathbf a^m-\mathbf b^m) \geq 0$ for all $a,b \in \mathbb{R}^N$ and $m >0$, which can be derived by using Young's inequality. One has $(\partial_t \mollifytime u h) \cdot \mollifytime \u h^m =  \frac 1{m+1} \partial_t|\mollifytime u h|^{m+1}$. This and partial integration with respect to the time variable $t$ extend the estimate to
\begin{align*}
\iint_{\Omega_T} & \alpha \xi \eta^2 \partial_t \mollifytime u h \cdot  \u^m \dx \dt \\
	&\geq \frac{1}{m+1} \iint_{\Omega_T} \alpha \xi \eta^2 \partial_t | \mollifytime u h |^{m+1} \dx \dt \\
	&= - c \iint_{\Omega_T} \eta^2 \partial_t( \alpha \xi ) |  \mollifytime u h |^{m+1}\dx \dt \\
	&= c \iint_{ \Omega \times (\tau-\sigma,\tau)}  \frac 1 \sigma \eta^2  \xi  |  \mollifytime u h |^{m+1} \dx \dt
		- c \iint_{ \Omega \times (t_1-\delta,t_1)} \frac{1}{\delta} \eta^2 \alpha  |  \mollifytime u h |^{m+1} \dx \dt
\end{align*}
with $c=c(m)$. Since $u \in L^{1+m}_\mathrm{loc}(\Omega_T)$, by letting $h \rightarrow 0$ this converges to
\begin{align*}
c \iint_{\Omega \times (\tau-\sigma,\tau)}  \frac 1 \sigma  \eta^2  \xi  |u|^{m+1} \dx \dt
		- c \iint_{\Omega \times (t_1-\delta,t_1)} \frac{1}{\delta} \eta^2 \alpha  |u|^{m+1} \dx \dt.
\end{align*}
Now consider the other term in the equation \eqref{eq:mollifiedEquation}. Since $\A(x,t,u,D\u^m) \in L^2(\Omega_T)$, it follows that
\begin{align*}
\lim_{h \rightarrow 0} &\iint_{\Omega_T} \alpha \xi \mollifytime {\A(x,t,u,D\u^m)} h \cdot D(\eta^2 \u^m) \dx \dt \\
	&=  \iint_{\Omega_T} \alpha \xi \A(x,t,u,D\u^m) \cdot D(\eta^2 \u^m) \dx \dt.
\end{align*}
Applying the growth conditions \eqref{growth} and Young's inequality, we get
\begin{align*}
-\iint_{\Omega_T} \alpha \xi &\A(x,t,u,D\u^m) \cdot D(\eta^2 \u^m) \dx \dt \\
	&\leq - \nu \iint_{\Omega_T} \alpha \xi \eta^2 |D\u^m|^2 \dx \dt 
		+ 2L \iint_{\Omega_T} \alpha \xi |D\u^m| |\eta| |D\eta| |\u^m| \dx \dt \\
	&\leq -\frac \nu 2 \iint_{\Omega_T} \alpha \xi \eta^2 |D\u^m|^2 \dx \dt 
		+ c \iint_{\Omega_T} \alpha \xi |D\eta|^2 |\u^m|^2 \dx \dt
\end{align*}
for $c=c(\nu,L)$. Since $\tau_1 < t_1- \delta$ in the mollifier, it follows that the right hand side of the equation~\eqref{eq:mollifiedEquation_loc} vanishes as $h\to 0$. Putting these estimates together, we obtain
\begin{align*}
\iint_{\Omega_T} \alpha \xi &\eta^2 |D\u^m|^2 \dx \dt
		+ \iint_{\Omega \times (\tau- \sigma, \tau)} \frac 1 \sigma \eta^2  \xi  |u|^{m+1} \dx \dt \\
	&\leq c \iint_{\Omega_T} \alpha \xi |D\eta|^2 |\u^m|^2 \dx \dt
		+c \iint_{\Omega \times (t_1-\delta,t_1)} \frac{1}{\delta} \eta^2 \alpha  |u|^{m+1} \dx \dt
\end{align*}
for a constant $c=c(m,\nu,L)$. Letting $\sigma \rightarrow 0$, this yields
\begin{align*}
\iint_{\Omega \times (t_1,\tau)} &\eta^2 |D\u^m|^2 \dx \dt
		+ \int_{\Omega \times \lbrace \tau \rbrace} \eta^2 |u|^{m+1} \dx \\
	&\leq \iint_{\Omega \times (t_1-\delta,\tau)} \xi \eta^2 |D\u^m|^2 \dx \dt
		+ \int_{\Omega \times \lbrace \tau \rbrace} \eta^2  \xi  |u|^{m+1} \dx  \\
	&\leq c \iint_{\Omega \times (t_1-\delta,\tau)} |D\eta|^2 |\u^m|^2 \dx \dt
		+c \iint_{\Omega \times (t_1-\delta,t_1)} \frac{1}{\delta} \eta^2  |u|^{m+1} \dx \dt,
\end{align*}
Choosing $\tau = t_2$ in the first term and taking the supremum over $(t_1,t_2)$ in the second term on the left hand side, we obtain the claimed inequality with $c_m=c(m,\nu,L)$. Tracing the initial appearance of $\frac 1{m+1}$ in the constants reveals that $\sup_{m \in (m_c,M)} c_m < \infty$.
\end{proof}


\begin{corollary}\label{kor:unifbound}
In the setting of Theorem \ref{theorem} and for any $\tilde \Omega \times (t_1,t_2) \Subset \Omega_T$, we have
\begin{align*}
\sup_{i \in \mathbb{N}}
	\bigg(
		\sup_{t \in (t_1,t_2)} &\int_{\tilde\Omega \times \lbrace t \rbrace} |u_i|^{m_i+1} \dx +
		\iint_{\tilde \Omega \times (t_1,t_2)} |\u_i^{m_i}|^2 \dx\dt \\
		&+\iint_{\tilde \Omega \times (t_1,t_2)} \big|D\u_i^{m_i}\big|^2 \dx\dt
	\bigg)
	<\infty.
\end{align*}
\end{corollary}

\begin{proof}
For $\tilde \Omega \times (t_1,t_2) \Subset \Omega_T$ choose $\eta \in C_0^\infty(\Omega)$ with $\eta = 1$ on $\tilde \Omega$, $\eta(x) = 0$ if $\dist(x, \partial\tilde \Omega) \geq \tfrac 1 2 \dist(\partial \Omega,\partial \tilde \Omega)$ and
$|D\eta| \leq \dist(\partial \Omega, \partial \tilde \Omega)^{-1}$. \\
Due to the weak convergence (\ref{assumption:weak-convergence}), the sequences $(\u_i^{m_i})$ and $(\u_i^{m_i+1})$ are bounded in $L^2(K \times I)$ and $L^1(K \times I)$ respectively, in each of the cases $m \geq 1, \; m<1$ and for any $K \times I \Subset \Omega_T$. This immediately gives the bound for the term $|\u_i^{m_i}|^2$. It further implies that the right hand side of the previous lemma is bounded. Hence the boundedness follows for the terms $|u_i|^{m_i+1}$ and $|D\u_i^{m_i}|^2$.
\end{proof}

\begin{corollary}\label{kor:weak-convergence}
In the setting of Theorem \ref{theorem}, for the function $u: \Omega_T \rightarrow \mathbb{R}^N$ from \eqref{assumption:weak-convergence} we have $\u^m \in L_{\loc}^2 \big( 0,T; W_{\loc}^{1,2}(\Omega,\mathbb{R}^N ) \big)$. Further, there is a subsequence, still denoted by $(\u_i^{m_i})$, such that
\begin{align*}
\u_i^{m_i} &\rightharpoonup \u^m \quad \text{ weakly in } L^2_{\loc} \big( 0,T; W_{\loc}^{1,2}(\Omega,\mathbb{R}^N ) \big)
\end{align*}
as $i \longrightarrow \infty$.
\end{corollary}

\begin{proof}
By the uniform bound of Corollary \ref{kor:unifbound}, we have that $(\u_i^{m_i})$ is a bounded sequence in $L^2_{\loc} \big( 0,T; W_{\loc}^{1,2}(\Omega, \mathbb{R}^N) \big)$. Hence there is a subsequence, still denoted by $(\u_i^{m_i}	)$, that converges weakly in that space to some limit function $\v^m$. As further $\u_i^{m_i} \rightharpoonup \v^m$ and, by \eqref{assumption:weak-convergence},  $\u_i^{m_i} \rightharpoonup \u^m$ both weakly in $L_{\loc}^2(\Omega_T,\mathbb R^N)$, the uniqueness of weak limits implies $\u^m=\v^m \in L^2_{\loc} \big( 0,T; W_{\loc}^{1,2}(\Omega, \mathbb{R}^N) \big)$.
\end{proof}

\subsection{Strong convergences}
\hfill \\[1ex]
Up next we improve the weak convergence to a strong convergence. We first treat the sequence of weak solutions and then the gradients. To obtain these results we use the energy estimate from the last section and the compactness result in Lemma~\ref{lem:simon-compactness}. Further, we need the following Lemma which is due to B\"ogelein et al.~\cite[Lemma 4.4]{Boegelein-Dietrich-Vestberg}.

\begin{lemma}\label{lem:interpolation}
Let $\beta \in (0,\infty)$, $p,q, \mu \in [1,\infty)$, $\theta := \max \lbrace 1, \beta p \rbrace$, $T>0$ and $\Omega \subset \mathbb R^n$ be a bounded domain and $X,Y$ be Banach spaces such that $X \subset L^q(\Omega, \mathbb R^N)$ and $L^{\mu'} (\Omega, \mathbb R^N) \subset X' \subset Y$ with compact embeddings $T:X \hookrightarrow L^q(\Omega,\mathbb R^N)$ and $S:L^{\mu'}(\Omega,\mathbb R^N) \hookrightarrow X'$ that are compatible in the sense that
\begin{align*}
\int_\Omega Tv \cdot w \dx
= \langle v, Sw \rangle
\end{align*}
for any $v \in X$ such that $Tv \in L^\mu(\Omega, \mathbb R^N)$, for any $w \in L^{\mu'}(\Omega, \mathbb R^N)$ and where $\langle \cdot, \cdot \rangle$ is the dual pairing of $X$ and $X'$. Then, for any $\eta >0$ there exists $M_\eta >0$ such that
\begin{align*}
\norm{ \tau_h f^\beta &- f^\beta }_{L^p(0,T-h; L^q(\Omega, \R^N ))} \\
	&\leq \norm{f^\beta}_{L^p(0,T;X)}^{\frac \beta{\beta+1}}
		\bigg[
			\eta \Big[
				\norm{f^\beta}_{L^p(0,T;X)}^{\frac 1{\beta+1}}
				+ \norm{f}_{L^\theta(0,T;L^{\mu'}(\Omega,\R^N))}^{\frac \beta{\beta+1}}
			\Big]  + M_\eta
			\norm{\tau_h f - f }_{L^\theta(0,T-h;Y)}^{\frac{ \beta}{\beta+1}}
		\bigg]
\end{align*}
for any $f \in L^\theta(0,T;L^{\mu'}(\Omega,\mathbb R^N))$ with $f^\beta \in L^p(0,T;X) \cap L(0,T;L^{\mu}(\Omega,\R^N))$.
\end{lemma}


\begin{lemma} \label{lem:strong-lp-convergence}
Let the assumptions of Theorem \ref{theorem} hold and $u: \Omega_T \rightarrow \mathbb{R}^N$ be the function from \eqref{assumption:weak-convergence}. Then there exists a subsequence, still denoted by $(\u_i^{m_i})$, such that
$$
\u_i^{m_i} \to \u^m \quad \text{strongly in } L_{\loc}^q(\Omega_T, \R^N),
$$
as $i \to \infty$ and for any $q<2$.
\end{lemma}

\begin{proof}
Let $0<t_1 < t_2 <T$, $K \Subset \Omega$. First we show that the sequence $(u_i)$ is uniformly equicontinuous in $C\big([t_1,t_2]; (W^{1,2}_0(K,\R^N))^{'}\big)$, where by $(W^{1,2}_0(K,\R^N))^{'}$ we denote the dual space of $W^{1,2}_0(K,\R^N)$. Let $\tau \in (t_1,t_2)$. For $h \in (t_1 , t_2 - \tau)$ and $\delta \in (0, \min \{ \tau, t_2 - \tau - h \})$ we define
\begin{equation*}
  \xi_\delta (t):=
  \begin{cases}
    0, &t< \tau - \delta,\\
    \frac{1}{\delta}(t - \tau + \delta), &t\in [\tau-\delta, \tau], \\
	1, &t\in (\tau,\tau + h), \\
	\frac{1}{\delta}(-t + h + \delta + \tau), &t\in [\tau+ h, \tau +h+\delta], \\
	0, &t > \tau + h + \delta.
  \end{cases}
\end{equation*}
Furthermore let $w \in W^{1,2}_0(K,\mathbb R^N)$. We want to test the mollified weak formulation \eqref{eq:mollifiedEquation_loc} with $\varphi = \xi_\delta w$ and small enough parameter $\tau_1>0$ in the mollification. Unless $m < (n-1)/(n+1)$, this function is clearly admissible. In the other case, we see that, if $n>2$, $w \in L^{2n/(n-2)}(\Omega)$ by the spatial Sobolev embedding and consequently, $w \in L^{2n/(n-2)}(\Omega_T)$. As $2n/(n-2) \geq (1+m_i)/m_i$ if and only if $m_i \geq (n-2)_+/(n+2)$, $\varphi$ is admissible also in this case. If $n=2$, we have $w \in L^r$ for all $1  \leq r < \infty$ and the same conclusion holds. \\
Hence, we are allowed to apply this test function and let the mollification parameter go to zero. This way, we obtain
\begin{align*}
\frac{1}{\delta} \int_{\tau-\delta}^{\tau} \int_\Omega u_i \cdot w\, \d x \d t
	- \frac{1}{\delta} &\int_{\tau+h}^{\tau +h+\delta} \int_\Omega u_i \cdot w\, \d x \d t \\
&=  \iint_{\Omega_T} \xi_\delta \mathbf A (x,t,u_i,D \u_i^{m_i}) \cdot D w\, \d x \d t.
\end{align*}
By passing to the limit $\delta \to 0$ it follows that
\begin{align*}
\int_\Omega [ u_i(\tau) - u_i(\tau+h) ] \cdot w\, \d x = \int_{\tau}^{\tau + h} \int_\Omega \mathbf A (x,t,u_i, D \u_i^{m_i}) \cdot D w\, \d x \d t.
\end{align*}
By the spatial Sobolev embedding and using the lower bound $m_i > m_c
$ one can see that $u_i(\cdot,\tau) \in L^\frac{2n}{n+2}(K,\mathbb R^N)$. Hence we can consider $u_i$ as an element in the space $(W^{1,2}_0(K,\mathbb R^N))^{'}$ for every fixed time. From the equation above we thus have
%
\begin{align*}
| \langle  u_i(\tau) - u_i(\tau+h) , w \rangle  |
	&\le \int_{\tau}^{\tau + h} \int_\Omega \mathbf | \A (x,t,u_i, D \u_i^{m_i}) | | D w |\, \d x \d t \\ 
	&\le  L \int_{\tau}^{\tau + h} \int_\Omega \mathbf |D \u_i^{m_i} | | D w |\, \d x \d t \\
	&\le L \| w \|_{W^{1,2}_0(K,\mathbb R^N)} \int_{\tau}^{\tau + h}\left( \int_K \mathbf |D \u_i^{m_i} |^2 \, \d x \right)^\frac12 \d t \\
	&\le L \| w \|_{W^{1,2}_0(K,\mathbb R^N)} h^\frac12 \left( \int_{t_1}^{t_2} \int_K |D \u_i^{m_i}|^2\, \d x \d t \right)^\frac12,
\end{align*}
where we used H\"older's inequality twice. The last integral on the right hand side is uniformly bounded by Corollary~\ref{kor:unifbound}. This implies the uniform equicontinuity of $(u_i)$ in $C\big([t_1,t_2], (W^{1,2}_0(K,\R^N))^{'}\big)$. \\ 
Now consider $\tilde K \subset K, \; I:= [t_1,t_2]$ and a cutoff function $\xi \in C_0^\infty(K)$ with $\xi =1$ in $\tilde K$ and $\xi \leq 1$. We recall the notation $(\tau_h f) (t) := f(t+h)$ for $h>0$. We apply H\"older's inequality and the interpolation Lemma \ref{lem:interpolation} with $p=2$, $\beta=m_i$, $X=W_0^{1,2}(\Omega,\mathbb R^N)$, $Y=X'$, $q= (1+m_i)/m_i$ and $\mu' = 1+ m_i$ to the function $f \equiv \xi^{m_i^{-1}} \u_i^{m_i}$. Note that the required compact embeddings exist since the exponents $m_i$ are always larger than the critical exponent $m_c$. The compatibility condition in Lemma \ref{lem:interpolation} is fulfilled by definition of the adjoint operator. This way, for arbitrary $\eta>0$ we obtain $M_\eta >0$ such that
\begin{align*}
\norm{ \tau_h & \u_i^{m_i} - \u_i^{m_i} }_{L^1( \tilde K \times I, \mathbb R^N)} \\
	&\leq c \norm{ \tau_h \u_i^{m_i} - \u_i^{m_i} }_{L^2(t_1,t_2-h;L^{q}( \tilde K, \mathbb R^N))} \\
	& \leq c \norm{ \xi (\tau_h \u_i^{m_i} - \u_i^{m_i}) }_{L^2(t_1,t_2-h;L^{q}( K, \mathbb R^N))} \\
	&\leq c \norm{ \xi \u_i^{m_i} }_{L^2(I,X)}^{m_i/(m_i+1)}
	\bigg[ \eta
		\Big[ \norm{ \xi \u_i^{m_i} }_{L^2(I,X)}^{1/(1+m_i)}
		+ \norm{ \xi^{m_i^{-1}} u_i }_{L^\theta(I,L^{1+m_i}(\Omega,\mathbb R^N))}
			^{m_i/(m_i+1)}
		\Big] \\
		&\quad+ M_\eta \norm{ \xi^{m_i^{-1}} (\tau_h u_i - u_i) }
			_{L^\theta (t_1,t_2-h; Y )}
			^{m_i/(m_i+1)}
	\bigg].
\end{align*}
Let $\varepsilon>0$. By estimating $\xi \leq 1$ and using the uniform bounds from Corollary \ref{kor:unifbound}, we can choose $\eta>0$ so small such that
\begin{align*}
\norm{ \tau_h \u_i^{m_i} - \u_i^{m_i} }_{L^1( \tilde K \times I, \mathbb R^N)} 
	\leq \varepsilon + M_\varepsilon
		\norm{ \tau_h u_i - u_i }_{L^\theta (t_0,t_1-h; Y )}^{m_i/(m_i+1)}.
\end{align*}
By using the uniform equicontinuity of  $(u_i)$ in $C\big([t_1,t_2], (W^{1,2}_0(K,\R^N))^{'}\big)$, it follows that
\begin{align*}
\lim_{h \downarrow 0} \norm{ \tau_h \u_i^{m_i} - \u_i^{m_i} }_{L^1( \tilde K \times I, \mathbb R^N)} 
	\leq \varepsilon
\end{align*}
uniformly in $i$. Since $\varepsilon>0$ was arbitrary, we have
\begin{align*}
\lim_{h \downarrow 0} \norm{ \tau_h \u_i^{m_i} - \u_i^{m_i} }_{L^1( \tilde K \times I, \mathbb R^N)} 
=0
\end{align*}
uniformly in $i$.


\noindent
This allows us to apply Lemma \ref{lem:simon-compactness} with $F=\lbrace \u_i^{m_i}: i \in \mathbb{N}\rbrace$, $p=1$, $X=W^{1,2}(\tilde K,\R^N)$ and $B=L^1(\tilde K,\R^N)$. This implies that there exists a subsequence, still denoted by $(\u_i^{m_i})$, such that $\u_i^{m_i} \longrightarrow \v^m$ strongly (and thus also weakly) in $L^1(\tilde K \times I,\R^N)$ for some limit function $\v^m$. The weak convergence to $\u^m$ in \eqref{assumption:weak-convergence} also holds in $L^1(\tilde K \times I,\R^N)$. By uniqueness of weak limits, it once again follows that $\v^m=\u^m \in L^2(\tilde K \times I,\R^N)$ (see Corollary \ref{kor:weak-convergence}). Clearly, this convergence is also fulfilled pointwise a.e. as $i \longrightarrow \infty$. \\
Since $\Omega$ is bounded, pointwise convergence a.e. implies convergence in measure. We recall that $\u_i^{m_i} \in L^2(\tilde K \times I,\R^N)$ and in the singular case $\u_i^{m_i} \in L^{\frac{m+1}{m}}(\tilde K \times I,\R^N)$ for large $i$ due to the weak convergence of the sequence and the Gagliardo-Nirenberg inequality in Lemma~\ref{lem:parabolic-sobolev}. In the case $m\ge 1$, for any $q<2$ and $\delta > 0$ we can write
%
\begin{align*}
\iint_{\tilde K \times I} | \u_i^{m_i} - \u^m |^q  \, \d x \d t
	= &\iint_{ (\tilde K \times I) \cap \{ | \u_i^{m_i} - \u^m | < \delta \} } | \u_i^{m_i} - \u^m |^q  \, \d x \d t \\
	&+ \iint_{(\tilde K \times I) \cap \{ | \u_i^{m_i} - \u^m | \ge \delta \}} | \u_i^{m_i} - \u^m |^q  \, \d x \d t.
\end{align*}
The first integral can be made small by choosing $\delta >0$ small and for the latter integral we can use H\"older's inequality
%
\begin{align*}
&\iint_{(\tilde K \times I) \cap \{ | \u_i^{m_i} - \u^m | \ge \delta \}} | \u_i^{m_i} - \u^m |^q  \, \d x \d t \\
&\hspace{15mm}\le |(\tilde K \times I) \cap \{ | \u_i^{m_i} - \u^m | \ge \delta \} |^\frac{2-q}{2} \Bigg( \iint_{\tilde K \times I } |\u_i^{m_i} - \u^m|^2 \, \d x \d t \Bigg)^\frac{q}{2}.
\end{align*}
The measure is small by the convergence in measure, and the latter integral is bounded by Corollary~\ref{kor:unifbound}. Thus it follows that $\u_i^{m_i} \longrightarrow \u^m$ strongly in $L^q(\tilde K \times I,\R^N)$ for all $q < 2$. In the singular case we can conclude the convergence for all $q<\frac{m+1}{m}$ with a similar argument.
\end{proof}

By using the Gagliardo-Nirenberg inequality in Lemma~\ref{lem:parabolic-sobolev} together with the same measure theoretic argument, we obtain slightly better convergences in the previous lemma which are stated in the following corollary.
\begin{corollary} \label{remark:strong-lp-convergence}
There is an $\varepsilon>0$ such that the convergences in Lemma~\ref{lem:strong-lp-convergence} hold for $L^{2+\varepsilon}_{\loc}(\Omega_T, \mathbb R^N)$ if $m\ge 1$, and $L^{\frac{m+1}{m}+\varepsilon}_{\loc}(\Omega_T, \mathbb R^N)$ if $m_c < m < 1$.
\end{corollary}

\begin{remark} \label{remark:wgradlim-identification}
From Corollary~\ref{kor:unifbound} it follows that there exists a weakly converging subsequence of $D\u_i^{m_i}$ in $L^2_{\loc}(\Omega_T,\R^{Nn})$. From Lemma~\ref{lem:strong-lp-convergence} it follows that the limit function is $D \u^m$, i.e. $D \u_i^{m_i} \rightharpoonup D\u^m$ weakly in $L^2_{\loc}(\Omega_T,\R^{Nn})$, as $i \to \infty$.
\end{remark}

\begin{lemma} \label{lem:strong-w1p-convergence}
Let the sequence $(u_i)$ be as in Lemma~\ref{lem:strong-lp-convergence}. By passing to another subsequence, we additionally have
$$
D \u_i^{m_i} \to D \u^m \quad \text{strongly in } L^{2}_{\loc}(\Omega_T,\R^{Nn})
$$
as $i \rightarrow \infty$.
\end{lemma}

\begin{proof}
Let $\xi \in C_0^\infty ( (0,T) ; [0,1] )$ and $\eta \in C_0^\infty(\Omega,\R_{\geq 0})$. By using the monotonicity~\eqref{monotone} we can estimate
\begin{align*}
\iint_{\Omega_T} \eta \xi &| D \u_i^{m_i} - D \u^m |^2\, \d x \d t \\
&\le c \iint_{\Omega_T} \eta \xi \left( \mathbf A(x,t,u_i,D\u_i^{m_i}) - \mathbf A(x,t,u,D\u^m) \right) \cdot \left( D \u_i^{m_i} - D\u^m \right)  \, \d x \d t \\
&= c \iint_{\Omega_T} \eta \xi \mathbf A(x,t,u_i,D\u_i^{m_i})  \cdot \left( D \u_i^{m_i} - D\u^m \right)  \, \d x \d t \\
&\hspace{8mm}-  c \iint_{\Omega_T} \eta \xi \mathbf A(x,t,u,D\u^{m})  \cdot \left( D \u_i^{m_i} - D\u^m \right)  \, \d x \d t \\
&=: c\, \mathrm I_i - c\, \mathrm {II}_i.
\end{align*}
From Remark~\ref{remark:wgradlim-identification} it follows that 
$$
\lim_{i \to \infty} \mathrm{II}_i = 0,
$$
so we need to focus only on the first term. For this one, we add and subtract mollified terms to obtain
%
\begin{align*}
\mathrm I_i &= c \iint_{\Omega_T} \eta \xi \mathbf A(x,t,u_i,D\u_i^{m_i})  \cdot \left( D \u_i^{m_i} - D\llbracket \u^m \rrbracket_h \right)  \, \d x \d t \\
&\hspace{8mm}+ c \iint_{\Omega_T}  \eta \xi \mathbf A(x,t,u_i,D\u_i^{m_i})  \cdot \left( D\llbracket \u^m \rrbracket_h - D\u^m \right)  \, \d x \d t \\
& =: c\, \mathrm I_i^{(1)} + c\, \mathrm I_i^{(2)}.
\end{align*}
The latter integral we can estimate as
\begin{align*}
|\mathrm I_i^{(2)}| &\le \| \eta \xi \mathbf A(x,t,u_i,D\u_i^{m_i}) \|_{L^2(\Omega_T)} \|\eta \xi ( D\llbracket \u^m \rrbracket_h - D\u^m ) \|_{L^2(\Omega_T)} \\
&\le L \| \eta \xi D\u_i^{m_i} \|_{L^2(\Omega_T)} \|\eta \xi ( D\llbracket \u^m \rrbracket_h - D\u^m ) \|_{L^2(\Omega_T)} \\
&\le c \|\eta \xi ( D\llbracket \u^m \rrbracket_h - D\u^m ) \|_{L^2(\Omega_T)},
\end{align*} 
by using growth conditions~\eqref{growth} and the uniform bound from Corollary~\ref{kor:unifbound}. Furthermore, the last expression on the right hand side converges to zero as $h \to 0$ by Lemma~\ref{lem:mollifier}. In order to estimate the term $\mathrm I_i^{(1)}$ we use the equation. We test the mollified weak formulation~\eqref{eq:mollifiedEquation_loc} with the test function $\varphi = \eta \xi (\u_i^{m_i} - \llbracket \u^m \rrbracket_h)$ and small enough $\tau_1 >0$ in the mollifications. \\
Unless $m$ is close to the critical parameter, this function is clearly admissible. If $m < {(n-1)}/{(n+1)}$, 
 we have to impose stronger assumptions such that the first integral in \eqref{eq:mollifiedEquation_loc} is finite. By the Gagliardo-Nirenberg inequality, see Lemma \ref{lem:parabolic-sobolev}, one can see that $u_i \in L^{1+m_i+\varepsilon}(\Omega_T, \mathbb R^N)$ for some $\varepsilon>0$. Thus it suffices if $\varphi$ is integrable to an exponent which is slightly smaller than $(1+m_i)/m_i$. For large $i$, this is also the case for $\u^m$ and consequently, also for $\mollifytime {\u^m} h$. \\
Now for the divergence part we obtain
\begin{align*}
\iint_{\Omega_T} \xi \llbracket &\mathbf A(x,t,u_i, D \u_i^{m_i}) \rrbracket_\lambda \cdot D [\eta (\u_i^{m_i} - \llbracket \u^m \rrbracket_h) ]   \, \d x \d t \\
&\xrightarrow{\lambda \to 0} \iint_{\Omega_T} \xi \mathbf A(x,t,u_i, D \u_i^{m_i})  \cdot D [\eta (\u_i^{m_i} - \llbracket \u^m \rrbracket_h) ]   \, \d x \d t \\
&= \mathrm I_i^{(1)} + \iint_{\Omega_T} \xi \mathbf A(x,t,u_i, D \u_i^{m_i})  \cdot (\u_i^{m_i} - \llbracket \u^m \rrbracket_h) \otimes D \eta   \, \d x \d t.
\end{align*}
On the other hand, we estimate the parabolic part. As in the proof of Lemma \ref{lem:caccioppoli}, we once again use the formula for the time derivative of the mollification from Lemma \ref{lem:mollifier} and the fact that $(a-b)(\mathbf{a}^m -\mathbf{b}^m ) \geq 0$ for all $a,b \in \mathbb{R}^N$ and $m>0$. This way, we obtain
\begin{align*}
\iint_{\Omega_T}  &\partial_t \llbracket u_i \rrbracket_\lambda \cdot \varphi  \, \d x \d t \\
& = \iint_{\Omega_T} \eta \xi \partial_t \llbracket u_i \rrbracket_\lambda \cdot (\u_i^{m_i} - \llbracket \u^m \rrbracket_h)  \, \d x \d t \\
&=\frac{1}{\lambda} \iint_{\Omega_T} \eta \xi ( u_i - \llbracket u_i \rrbracket_\lambda) \cdot (\u_i^{m_i} - \llbracket \u_i \rrbracket_\lambda^{m_i}) \, \d x \d t \\
&\hspace{8mm}+ \iint_{\Omega_T} \eta \xi \left( \partial_t\llbracket u_i  \rrbracket_\lambda \cdot \llbracket \u_i \rrbracket_\lambda^{m_i}  - \partial_t \llbracket u_i \rrbracket_\lambda \cdot   \llbracket \u^m \rrbracket_h \right) \, \d x \d t \\
&\ge \iint_{\Omega_T} \eta \xi \left( \partial_t\llbracket u_i \rrbracket_\lambda \cdot \llbracket \u_i \rrbracket_\lambda^{m_i}  - \partial_t \llbracket u_i \rrbracket_\lambda \cdot  \llbracket \u^m \rrbracket_h \right) \, \d x \d t \\
&= \iint_{\Omega_T} \eta \left( \frac{- 1}{m_i + 1} \partial_t \xi  |\llbracket u_i \rrbracket_\lambda|^{m_i+1}  + \xi  \llbracket u_i \rrbracket_\lambda \cdot \partial_t \llbracket \u^m \rrbracket_h + \partial_t \xi \llbracket u_i \rrbracket_\lambda \cdot \llbracket \u^m \rrbracket_h \right) \, \d x \d t \\
&\xrightarrow{\lambda \to 0} \iint_{\Omega_T} \eta \left( \frac{- 1}{m_i + 1} \partial_t \xi  |u_i|^{m_i+1}  + \xi u_i \cdot \partial_t \llbracket \u^m \rrbracket_h + \partial_t \xi u_i \cdot \llbracket \u^m \rrbracket_h \right) \, \d x \d t.
\end{align*}
By collecting the previous estimates we have 
\begin{align*}
\mathrm I_i^{(1)}
	\leq \iint_{\Omega_T} \eta
		&\left[ \partial_t \xi
			\left( \frac{1}{m_i + 1}  |u_i|^{m_i+1}
			- u_i \cdot \llbracket \u^m \rrbracket_h \right)
			- \xi u_i  \cdot \partial_t \llbracket \u^m \rrbracket_h \right]
		\, \d x \d t\\
	&- \iint_{\Omega_T} \xi \mathbf A(x,t,u_i, D \u_i^{m_i})  \cdot (\u_i^{m_i} - \llbracket \u^m \rrbracket_h) \otimes D \eta   \, \d x \d t.
\end{align*}
For the first integral on the right hand side we use the convergence properties of $u_i$ (and pass to a subsequence, if needed). For the second integral, we use Cauchy-Schwarz combined with the growth condition \eqref{growth} and the uniform bound for $|D\u_i^{m_i}|^2$. Together, this implies
\begin{align*}
\limsup_{i\to \infty} \mathrm I_i^{(1)}
	\leq \iint_{\Omega_T} \eta &\left[ \partial_t \xi \left( \frac{1}{m + 1}  |u|^{m+1} - u \cdot \llbracket \u^m \rrbracket_h \right) - \xi u  \cdot  \partial_t \llbracket \u^m \rrbracket_h \right] \, \d x \d t\\
	&+ c \|\xi (\u^{m} - \llbracket \u^m \rrbracket_h) \otimes D \eta \|_{L^2(\Omega_T)}.
\end{align*}
We further estimate the last term in the first integral. We add and subtract $\mollifytime {\u^m}h ^{\frac 1 m}$ and obtain similarly to before that
\begin{align*}
- &\iint_{\Omega_T} \eta \xi u  \cdot \partial_t \llbracket \u^m \rrbracket_h \, \d x \d t \\
&= - \iint_{\Omega_T} \eta \xi \left( \frac{1}{h} (u-\llbracket \u^m \rrbracket_h^\frac{1}{m})  \cdot  ( \u^m - \llbracket \u^m \rrbracket_h ) +  \llbracket \u^m \rrbracket_h^\frac{1}{m} \cdot  \partial_t \llbracket \u^m \rrbracket_h  \right)\, \d x \d t \\
&\le - \iint_{\Omega_T} \eta \xi   \llbracket \u^m  \rrbracket_h^\frac{1}{m}  \cdot \partial_t \llbracket \u^m \rrbracket_h  \, \d x \d t \\
&=\frac{m}{m+1} \iint_{\Omega_T} \eta \partial_t \xi   |\llbracket \u^m \rrbracket_h|^\frac{m+1}{m}  \, \d x \d t.
\end{align*}
By combining this with the earlier estimates and using the obtained convergence properties, we eventually arrive at 
\begin{align*}
\limsup_{i\to \infty} &\iint_{\Omega_T} \eta \xi | D \u_i^{m_i} - D \u^m |^2\, \d x \d t \\
&\leq c  \|\eta \xi ( D\llbracket \u^m \rrbracket_h - D\u^m ) \|_{L^2(\Omega_T)} \\
&\hspace{5mm}+ \iint_{\Omega_T} \eta \partial_t \xi \left(\frac{1}{m + 1}  |u|^{m+1} + \frac{m}{m+1} |\llbracket \u^m \rrbracket_h|^\frac{m+1}{m} - u  \cdot  \llbracket \u^m \rrbracket_h \right) \, \d x \d t \\
&\hspace{5mm}+ \|\xi (\u^{m} - \llbracket \u^m \rrbracket_h) \otimes D \eta \|_{L^2(\Omega_T)}   \\
&\longrightarrow 0
\end{align*}
as $h \to 0$ by Lemma~\ref{lem:mollifier}, from which the claim follows.
\end{proof}
Then finally we complete the proof of Theorem~\ref{theorem} with the following Lemma, which states that the limit function $u$ solves the corresponding limit problem.
\begin{lemma}\label{lem:limit_function_is_sol}
Let the sequences $(u_i)$ and $(m_i)$ be as in Lemmas~\ref{lem:strong-lp-convergence} and~\ref{lem:strong-w1p-convergence}. Then the limit function $u$ satisfies the problem with the limit parameter $m$, i.e.,
\begin{align*}
\lim_{i\to \infty} \Bigg( &\iint_{\Omega_T} \big[ - u_i  \cdot \partial_t \varphi + \mathbf A (x,t,u_i,D\u_i^{m_i}) \cdot D \varphi \big] \, \d x \d t \Bigg) \\
&= \iint_{\Omega_T} \big[ - u \cdot  \partial_t \varphi + \mathbf A (x,t,u,D\u^{m}) \cdot D \varphi \big] \, \d x \d t = 0,
\end{align*}
for any test function $\varphi \in C_0^{\infty}(\Omega_T,\R^N)$. 
\end{lemma}

\begin{proof}
Lemma~\ref{lem:strong-lp-convergence} and Corollary~\ref{remark:strong-lp-convergence} imply that $u_i \to u$ strongly in $L^1_{\loc}(\Omega_T,\R^N)$, from which it follows immediately that the parabolic part converges. By the same argument on a subsequence level we have a.e. pointwise convergence for the sequence $(u_i)$. Also, the strong convergence for the gradients $D\u_i^{m_i}$ from Lemma~\ref{lem:strong-w1p-convergence} yields a.e. pointwise convergence for a subsequence of the gradients. That is, $(u_i, D\u_i^{m_i}) \to (u,D\u^m)$ a.e. in any relatively compact subdomain of $\Omega_T$. By continuity of $\mathbf A$ with respect to the pair of the last two variables it follows that $\mathbf A(x,t,u_i,D\u_i^{m_i}) \to \mathbf A(x,t,u,D\u^m)$ a.e. as well. Since we are working in bounded domains, it follows that pointwise convergence a.e. implies convergence in measure. By a measure theoretic argument as in the final step in Lemma~\ref{lem:strong-lp-convergence} and by Corollary~\ref{kor:unifbound} it follows that $\mathbf A(x,t,u_i,D\u_i^{m_i}) \to \mathbf A(x,t,u,D\u^m)$ in $L^q_{\loc}(\Omega_T,\R^{Nn})$ for any $q<2$. This implies the convergence of the divergence part, which completes the proof.
\end{proof}

\begin{remark}\label{rem:whole-seq}
	So far we only showed convergence of a subsequence. However, we apply a similar argument as in \cite[p. 43]{Kinnunen-Parviainen}: For every subsequence of $\u_i^{m_i}$ we find another subsequence converging to some limit $\v^m$. By assumption \eqref{assumption:weak-convergence}, $\u_i^{m_i} \rightharpoonup \u^m$ weakly in $L^2_{\loc}(\Omega_T,\mathbb R^N)$. As weak limits are unique, we have $\u^m=\v^m$ for each limit. Since every subsequence has a converging subsequence with limit $\u^m$, the original sequence must converge and its limit must be $\u^m$.
\end{remark}

\section{Stability for the Cauchy-Dirichlet problem}

\noindent
We turn our attention to the Cauchy-Dirichlet problem. First, we inspect why the conditions for the boundary datum $g$ were chosen in \eqref{assumption:g}. These imply the following Lemma, which collects all the properties that we require for our boundary function $g$.

\begin{lemma}\label{lem:g_conditions}
For $g: \Omega_T \rightarrow \mathbb R^N$ satisfying the conditions \eqref{assumption:g}, there holds
\begin{align*}
g \in C \big( [0,T],L^{1+m_i}(\Omega,\mathbb R^N) \big)
\quad\text{ with }\quad
\g^{m_i} \in L^2(0,T;W^{1,2}(\Omega,\mathbb R^N))
\end{align*}
and additionally $\partial_t \g^{m_i} \in L^{\frac{m_i+1}{m_i}}(\Omega_T, \mathbb R^N)$ 
for all $i \in \mathbb N$. Further, we have the bounds
\begin{align*}
\sup_{i \in \mathbb{N}}
	\iint_{\Omega_T} \big( |g|^{1+m_i} + |\g^{m_i}|^2 + |D\g^{m_i}|^2 + |\partial_t \g^{m_i}|^{\frac {{m_i}+1}{m_i}} \big) \dx \dt
	<\infty
\end{align*}
and
\begin{align*}
\sup_{i \in \mathbb{N}}
\sup_\tau \int_{\Omega \times \lbrace \tau \rbrace}  |g|^{m_i+1} \dx  < \infty.
\end{align*}
For some $\varepsilon>0$ and as $i \longrightarrow \infty$, we have
\begin{align*}
\hspace{3cm}
\g^{m_i} &\longrightarrow \g^m \quad &&\text{ strongly in } L^{2+\varepsilon}(\Omega_T,\mathbb R^N), \\
D\g^{m_i} &\longrightarrow D\g^m \quad &&\text{ strongly in } L^{2+\varepsilon}(\Omega_T,\mathbb R^{Nn}), \\
\partial_t \g^{m_i} &\longrightarrow \partial_t \g^m \quad &&\text{ strongly in } L^{(1+m)/m+\varepsilon}(\Omega_T,\mathbb R^N)
\hspace{3cm}
\end{align*}
and also
\begin{align*}
\hspace{3cm}
\g^{1+m_i} &\longrightarrow \g^{1+m} \quad &&\text{ strongly in } L^{1+\varepsilon}(\Omega_T,\mathbb R^N), \\
\g^{m_i} &\longrightarrow \g^m \quad &&\text{ strongly in } L^{(1+m)/m+\varepsilon}(\Omega_T,\mathbb R^{N}).
\hspace{3cm}
\end{align*}
\end{lemma}

\begin{proof}
We illustrate the proof for the spatial gradients. The other elements require analogous arguments. We start by rewriting
\begin{align}\label{eq:g_grad_rewriting}
D\g^{m_i}
	= D(\g^{\tilde m})^{{m_i}/{\tilde m}}
	= \tfrac{m_i}{\tilde m} \g^{m_i-\tilde m} D\g^{\tilde m}.
\end{align}
Note that since $\beta > 2 \frac{m}{\tilde m}$, there is $\varepsilon>0$ such that $\beta> (2+\varepsilon) \frac{m}{\tilde m}$. We can thus estimate using Young's inequality
\begin{align*}
|D\g^{m_i}|^{2+\varepsilon}
	&\leq c |\g^{m_i-\tilde m}|^{2+\varepsilon}
		|D\g^{\tilde m}|^{2+\varepsilon}
	\leq c |\g^{m_i-\tilde m}|
		^{\frac{({2+\varepsilon})\beta}{\beta-2-\varepsilon}}
		+ c |D\g^{\tilde m}|^\beta.
\end{align*}
One can calculate that $(m_i-\tilde m)(2+\varepsilon)\beta/(\beta-2-\varepsilon) < \tilde m \beta$ using the property $\beta> (2+\varepsilon) \frac{m}{\tilde m}$. It follows that
\begin{align*}
|D\g^{m_i}|^{2+\varepsilon}
	&\leq c (1+|g|)^{\tilde m \beta} + c |D\g^{\tilde m}|^\beta.
\end{align*}
The right hand side is integrable, so we conclude that $D\g^{m_i} \in L^2(\Omega_T,\mathbb R^{Nn})$ for all $i$. Since the right hand side is independent of $i$, the estimate
\begin{align*}
\sup_{i \in \mathbb N} \iint_{\Omega_T} |D\g^{m_i}|^2 \dx \dt < \infty
\end{align*}
follows at once. Further, \eqref{eq:g_grad_rewriting} implies that $D\g^{m_i}$ converges to $D\g^{m}$ pointwise a.e. in $\Omega_T$. Combining this with the fact that $D\g^{m_i}$ forms a bounded sequence in $L^{2+\varepsilon}(\Omega_T, \mathbb R^{Nn})$, the strong convergence follows, for a slightly smaller $\varepsilon>0$. \\
The elements regarding $\g^{1+m_i}$ or the partial time derivative of $\g^{m_i}$ also make use of the properties regarding $\gamma > 1+m$.
\end{proof}

\subsection{Energy estimate}
\hfill \\[1ex]
To prove Theorem \ref{theorem:global}, one could repeat the same procedure as in the local setting, possibly using global higher integrability results in \cite{Moring-Scheven-Schwarzacher-Singer} to improve the convergence, at least in the degenerate case. However, we have chosen to use the local result. We use the weak convergence of $\u_i^{m_i}$ and the boundedness of $\u_i^{1+m_i}$, in the singular case. These will emerge from the following energy estimate.

\begin{lemma}\label{lem:caccioppoli_global}
Let $u$ be a weak solution to Equation 
\eqref{equation:global-PME} with exponent $m$ and with initial and boundary values $g$ satisfying \eqref{assumption:g}, in the sense of Definition \ref{def:global_weak_solution},
where the vector field $\mathbf A$ satisfies the growth conditions~\eqref{growth}. Then we have the estimate
\begin{align*}
\sup_\tau \int_{\Omega \times \lbrace \tau \rbrace} & |u|^{m+1} \dx 
+\iint_{\Omega_T} \big(  |\u^ {m}|^2 + |D\u^{m}|^2  \big) \dx \dt \\
	\leq \;\; 
		&C_m \, \sup_\tau \int_{\Omega \times \lbrace \tau \rbrace}  |g|^{m+1} \dx + C_m \iint_{\Omega_T} \big( 
		 |\g^m|^2 + |D\g^{m}|^2 + |\partial_t \g^{m}|^{\frac {{m}+1}{m}} \big) \dx \dt
\end{align*}
for a constant $C_m= C_m{(T,\nu,L,\mathrm{diam}(\Omega))} >0$ with
\begin{align*}
\sup_{m \in (m_c,M)} C_m < \infty
\end{align*}
for any given $M>m_c$.
\end{lemma}

\begin{remark} \label{rem:caccioppoli_glob}
The proof below is based on the mollified equation~\eqref{eq:mollifiedEquation}. However, when using just the assumptions on $u$ given in Definition~\ref{def:global_weak_solution}, then one should use the original equation~\eqref{def:weak_solution} and test function $\varphi = \alpha (\mollifytime{\u^m}{\bar h} - \mollifytime{\g^m}{\bar h} )$ instead. This will result in the same estimate. Observe that here we need that $\partial_t \mollifytime{\g^m}{\bar h} \to \partial_t \g^m$ in $L^\frac{m+1}{m}(\Omega_T)$ as $h \searrow 0$, for which reason we define the reverse mollifications in this case as
\begin{align*}
\mollifytime{\g^m}{\bar h} (\cdot,t) := e^\frac{t-T}{h}\g^m(\cdot,T) + \frac 1 h \int_t^T e^\frac{t-s}{h} \g^m(\cdot,s)\, \d s
\end{align*}
and
\begin{align*}
\mollifytime{\u^m}{\bar h} (\cdot,t) := e^\frac{t-T}{h}\g^m(\cdot,T) + \frac 1 h \int_t^T e^\frac{t-s}{h} \u^m(\cdot,s)\, \d s.
\end{align*}
%
%
\end{remark}

\proof For $\sigma >0$ and $\tau \in (0,T]$ define $\alpha=\alpha_\sigma \in W^{1,\infty}([0,T], [0,1])$ with
\begin{align*}
\alpha(t) := \begin{cases}
0 & 0 \leq t < \sigma \\
\tfrac 1 \sigma (t-\sigma) & \sigma \leq t < 2\sigma \\
1 & 2\sigma \leq t < \tau - \sigma \\
\tfrac 1 \sigma (\tau -t) & \tau - \sigma \leq t < \tau \\
0 & \tau \leq t \leq T.
\end{cases}
\end{align*}
Choose $\varphi = \alpha (\u^m-\g^m) \in L^2 \big( 0,T; W_0^{1,2}(\Omega,\mathbb{R}^N)\big)$ as testing function in the mollified equation (\ref{eq:mollifiedEquation}). By using $\partial_t \mollifytime u h = \tfrac 1 h \big( u-\mollifytime u h \big)$ from Lemma \ref{lem:mollifier}, we have
\begin{align*}
\iint_{\Omega_T}  \partial_t & \mollifytime u h \alpha(\g^m-\u^m) \dx \dt\\
	&= \iint_{\Omega_T} \partial_t \mollifytime u h \alpha(\g^m-\mollifytime \u h^m) \dx \dt
		+\iint_{\Omega_T} \partial_t \mollifytime u h \alpha(\mollifytime \u h^m-\u^m)  \dx \dt \\
	&= \iint_{\Omega_T} \partial_t \mollifytime u h \alpha(\g^m-\mollifytime \u h^m) \dx \dt
		- \frac 1 h \iint_{\Omega_T} \alpha (u-\mollifytime u h) (\u^m-\mollifytime \u h^m) \dx \dt.
\end{align*}
Since for all $m>0$ and all $a,b \in \mathbb{R}^N, \; (a-b)(\a^m-\mathbf{b}^m) \geq 0$, the second integral on the right hand side can be estimated by zero. Thus,
\begin{align*}
\iint_{\Omega_T}  \partial_t & \mollifytime u h \alpha(\g^m-\u^m) \dx \dt
	\leq \iint_{\Omega_T} \partial_t \mollifytime u h \alpha(\g^m-\mollifytime \u h^m) \dx \dt.
\end{align*}
One can easily see that $(\partial_t v)\v^m = \frac{1}{m+1} \partial_t|v|^{m+1}$. Using integration by parts with respect to the time variable $t$, we have the identity
\begin{align*}
&\iint_{\Omega_T} (\partial_t \mollifytime u h) \alpha  (\g^m-\mollifytime \u h^m) \dx \dt \\
	&\quad= - \iint_{\Omega_T}\tfrac{1}{m+1}  \alpha  \partial_t |\mollifytime u h|^{m+1} \dx \dt + \iint_{\Omega_T} \alpha  \partial_t \mollifytime u h \g^m \dx \dt \\
	&\quad=\iint_{\Omega_T} \tfrac{1}{m+1}\partial_t \alpha  |\mollifytime u h|^{m+1} \dx \dt 
	- \iint_{\Omega_T} (\partial_t \alpha)  \mollifytime u h \g^m \dx \dt
	- \iint_{\Omega_T} \alpha  \mollifytime u h \partial_t \g^m \dx \dt.
\end{align*}
Since $u \in L^{m+1}({\Omega_T})$, it follows from Lemma \ref{lem:mollifier} that $\mollifytime u h \longrightarrow u$ in $L^{m+1}({\Omega_T})$ as $h \to 0$. Thus, by passing to the limit $h \to 0$ the right hand side equals
\begin{align*}
&\iint_{\Omega_T} \tfrac 1 {m+1} \partial_t  \alpha  |u|^{m+1} \dx \dt
	-\iint_{\Omega_T} \partial_t \alpha u \g^m  \dx \dt
	- \iint_{\Omega_T} \alpha u \partial_t \g^m \dx \dt \\
	= &\iint_{\Omega_T} \partial_t \alpha \Big( \tfrac 1 {m+1}
		\big(
			|u|^{m+1} - |g|^{m+1}
		\big)
		- \g^m(u-g) \Big) \dx \dt \\
	&\quad
		- \iint_{\Omega_T} \alpha u \partial_t \g^m \dx \dt
		- \iint_{\Omega_T} \partial_t \alpha \tfrac m{m+1} |g|^{m+1} \dx \dt \\
	= &\iint_{\Omega_T} \partial_t \alpha \Big(
		\tfrac 1{m+1} \big( |u|^{m+1} - |g|^{m+1} \big) 
		-\g^m(u-g) \Big) \dx \dt
		+\iint_{\Omega_T} \alpha \partial_t \g^m(g-u) \dx \dt,	
\end{align*}
where in the last step $\tfrac m{m+1} \partial_t |g|^{m+1} = (\partial_t \g^m)g$ was used. The first integral on the right hand side contains the boundary term $I(u,g)$. Thus,
\begin{align}\label{eq:EnergyEstimate_LHS}
\begin{aligned}
\lim_{h \downarrow 0} &\iint_{\Omega_T}  \partial_t  \mollifytime u h \alpha(\g^m-\u^m) \d x \d t \\
	& \quad \leq \iint_{\Omega_T} \partial_t \alpha I(u,g) \dx \dt
	+\iint_{\Omega_T} \alpha \partial_t \g^m(g-u) \dx \dt.
\end{aligned}
\end{align}
Since $\alpha(t) = 0$ for $t \in [0,\sigma]$, it follows that the right hand side of~\eqref{eq:mollifiedEquation} vanishes as $h\to 0$. Let us now inspect the divergence term of the mollified equation (\ref{eq:mollifiedEquation}). Since we have $\A(x,t,u,D\u^m) \in L^2(\Omega_T, \mathbb R^{Nn})$, by Lemma \ref{lem:mollifier} it follows that
\begin{align*}
\lim_{h \downarrow 0}
&\iint_{\Omega_T} \alpha \mollifytime  {\A(x,t,u,D\u^m)} h \cdot D  (\u^m-\g^m)  \dx \dt \\
	&= \iint_{\Omega_T} \alpha \A(x,t,u,D\u^m) \cdot D (\u^m-\g^m)\dx \dt.
\end{align*}
For the integral on the right hand side we have
\begin{align*}
 \iint_{\Omega_T}  \alpha &\A(x,t,u,D\u^m) \cdot D (\u^m- \g^m)\dx \dt \\
 	&\geq \nu \iint_{\Omega_T}  \alpha |D\u^m|^2 \dx \dt - \int_{\Omega_T}  \alpha \A(x,t,u,D\u^m) \cdot D\g^m \dx \dt.
\end{align*}
Rearranging the second term of the right hand side, using Cauchy-Schwarz inequality, the structure conditions \eqref{growth} and the estimate \eqref{eq:EnergyEstimate_LHS} yields
\begin{align*}
\nu \iint_{\Omega_T}  \alpha |D\u^m|^2 \dx \dt
	&\leq L \iint_{\Omega_T}  \alpha |D\u^m| |D\g^m| \dx \dt 
		+\iint_{\Omega_T} \alpha \partial_t \g^m(g-u) \dx \dt \\
		& \quad + \iint_{\Omega_T} \partial_t \alpha \, I(u,g) \dx \dt.
\end{align*}
We now pass to the limit $\sigma \downarrow 0$ and write $\Omega_\tau := \Omega \times (0,\tau)$. We recall that $u,g \in C^0 \big([0,T],L^{m+1}(\Omega,\mathbb R^N) \big)$ and use the initial boundary condition \eqref{initial-boundary} to get
\begin{align*}
\nu \iint_{\Omega_\tau} |D\u^m|^2 \dx \dt
	&\leq L \iint_{\Omega_\tau} |D\u^m| |D\g^m| \dx \dt 
		+\iint_{\Omega_\tau} \partial_t \g^m(g-u) \dx \dt \\
		& \quad - \int_{\Omega \times \lbrace \tau \rbrace}  I(u,g) \dx
\end{align*}
for any $\tau \in (0,T]$. For the first two terms on the right hand side we use Young's $\varepsilon$-inequality and obtain
\begin{align*}
\int_{\Omega \times \lbrace \tau \rbrace} & I(u,g) \dx
	+ \nu \iint_{\Omega_\tau} |D\u^m|^2 \dx\dt \\
	&\leq \frac{L \varepsilon_1}{2} \iint_{\Omega_\tau} |D\u^m|^2 \dx \dt
		+ \frac{L}{2 \varepsilon_1} \iint_{\Omega_\tau} |D\g^m|^2 \dx \dt \\
	&\quad+ \frac{m}{(m+1)\varepsilon_2} \iint_{\Omega_\tau} |\partial_t \g^m|^{\frac{m+1}m} \dx \dt
		+ \frac{2^m \varepsilon_2}{m+1} \iint_{\Omega_\tau} \big( |g|^{m+1} + |u|^{m+1} \big) \dx\dt.
\end{align*}
We choose $\varepsilon_1>0$ small so that we can absorb the first term on the right hand side into the left hand side. This way, we get
\begin{align*}
\int_{\Omega \times \lbrace \tau \rbrace} &I(u,g) \dx
	+ \iint_{\Omega_\tau} |D\u^m|^2 \dx \dt \\
	&\leq C \iint_{\Omega_T} |D\g^m|^2 \dx \dt 
		+ C \iint_{\Omega_T} |\partial_t \g^m|^{\frac{m+1}m} \dx \dt \\
	&\quad+ C \varepsilon_2 \iint_{\Omega_T} \big( |g|^{m+1} + |u|^{m+1} \big) \dx\dt,
\end{align*}
where $C=C(m,\nu,L)$. Since the right hand side is now independent of $\tau$, we choose $\tau=T$ in the second term on the left hand side. We also take the supremum over
 $\tau \in [0,T]$ and use the estimates
\begin{align*}
C(m) \big(|u|^{m+1} - |g|^{m+1} \big)
	\leq I(u,g)
	\leq C(m) \big( |u|^{m+1} + |g|^{m+1} \big),
\end{align*}
which can be derived from Young's inequality. It follows that
\begin{align*}
\sup_\tau &\int_{\Omega \times \lbrace \tau \rbrace} |u|^{m+1} \dx 
	+ \iint_{\Omega_T} |D\u^m|^2 \dx \dt \\
	&\leq C \iint_{\Omega_T} |D\g^m|^2 \dx \dt 
		+ C \iint_{\Omega_T} |\partial_t \g^m|^{\frac{m+1}m} \dx \dt \\
	&\quad+ C \sup_\tau \int_{\Omega \times \lbrace \tau \rbrace}  |g|^{m+1} \dx 
	+ C \iint_{\Omega_T} |g|^{m+1} \dx \dt
	+ C \varepsilon_2 \sup_\tau \int_{\Omega \times \lbrace \tau \rbrace} |u|^{m+1}  \dx.
\end{align*}
Choosing $\varepsilon_2$ small enough, we can absorb another term into the left hand side and get
\begin{align*}
\sup_\tau &\int_{\Omega \times \lbrace \tau \rbrace} |u|^{m+1} \dx 
	+ \iint_{\Omega_T} |D\u^m|^2 \dx \dt \\
	&\leq C \sup_\tau \int_{\Omega \times \lbrace \tau \rbrace}  |g|^{m+1} \dx 
	+ C \iint_{\Omega_T} \big( |g|^{m+1} + |D\g^m|^2
		+|\partial_t \g^m|^{\frac{m+1}m} \big) \dx \dt.
\end{align*}
We have $C=C(m,T,\nu,L)$ for now. For the term containing $|\u^m|^2$ use Poincar\'e's inequality:
\begin{align*}
\iint_{\Omega_T} |\u^m|^2 \dx \dt
	&\leq 2\iint_{\Omega_T} |\u^m-\g^m|^2 \dx \dt
		+ 2\iint_{\Omega_T} |\g^m|^2 \dx \dt \\
	&\leq  2\frac {\mathrm{diam}\,(\Omega)^2}2 \iint_{\Omega_T} |D\u^m-D\g^m|^2 \dx \dt
		+ 2\iint_{\Omega_T} |\g^m|^2 \dx \dt \\
	&\leq C \iint_{\Omega_T} |D\u^m|^2 \dx \dt
		+ C \iint_{\Omega_T} |D\g^m|^2 \dx \dt
		+ 2 \iint_{\Omega_T} |\g^m|^2 \dx \dt,
\end{align*}
where $C=C(\mathrm{diam}\,\Omega)$. This yields the desired estimate with $C_m=C_{({m},T,\nu,L,\mathrm{diam}(\Omega))} >0$ and
\begin{align*}
\sup_{m \in (m_c,M)} C_m < \infty
\end{align*}
for any given $M>m_c$. \qed

\begin{corollary}
Let $u_i$ be a weak solution to Equation 
\eqref{equation:global-PME} with exponent $m_i$ and with initial and boundary values $g$ satisfying \eqref{assumption:g}, in the sense of Definition \ref{def:global_weak_solution},
where the vector field $\mathbf A$ satisfies the growth conditions~\eqref{growth}.
Then we can bound
\begin{align*}
\sup_{i \in \mathbb{N}} \bigg(
	\sup_\tau \int_{\Omega \times \lbrace \tau \rbrace}  |u_i|^{m_i+1} \dx 
	+ \iint_{\Omega_T} |\u_i^{m_i}|^2 + |D\u_i^{m_i}|^2 \dx \dt
\bigg)
	< \infty.
\end{align*}
\end{corollary}

\begin{remark}
In the degenerate case, one can adapt the proof slightly, which in turn admits weaker assumptions for the initial and boundary datum $g$. More specifically, one can use Young's inequality with the exponent $2m$ instead of $1+m$ for the expression $\partial_t \g^m (g-u)$. By further absorbing the $|\u^m|^{2}$ term to the left hand side after the usage of the Poincar\'e-inequality, one has proven the following variation of the previous Lemma \ref{lem:caccioppoli_global}:
\end{remark}

\begin{lemma}
Let $u$ be a weak solution to Equation 
\eqref{equation:global-PME} with exponent $m$ and with initial and boundary values $g$ satisfying \eqref{assumption:g}, in the sense of Definition \ref{def:global_weak_solution},
where the vector field $\mathbf A$ satisfies the growth conditions~\eqref{growth}.
For $m\geq 1$, we then have the estimate
\begin{align*}
\sup_\tau &\int_{\Omega \times \lbrace \tau \rbrace}  |u|^{m+1} \dx 
+\iint_{\Omega_T} \big(  |\u^ {m}|^2 + |D\u^{m}|^2  \big) \dx \dt \\
	&\; \leq 
		C_m \sup_\tau \int_{\Omega \times \lbrace \tau \rbrace}  |g|^{m+1} \dx 
		+C_m \iint_{\Omega_T} \big( |\g^m|^2 + |D\g^{m}|^2 + |\partial_t \g^{m}|^{\frac {2m}{2m-1}} \big) \dx \dt
\end{align*}
for a constant $C_m= C_m{(T,\nu,L,\mathrm{diam}(\Omega))} >0$ with
\begin{align*}
\sup_{m \in (m_c,M)} C_m < \infty
\end{align*}
for any given $M>m_c$.
\end{lemma}

\subsection{Strong convergences}
\hfill \\[1ex]
These energy estimates imply weak convergence and also allow us to apply the local result to achieve local convergence. These can immediately be transferred to get stronger convergence for a lower exponent. Further, by using Gagliardo-Nirenberg inequality, a stronger convergence for the solutions themselves can be attained. We collect this information in the following Lemma.

\begin{lemma}\label{lem:global_convergences_immediate}
In the setting of Theorem \ref{theorem:global}
there is a subsequence, still denoted by $(u_i)$, and a measurable function $u: \Omega_T \rightarrow \mathbb{R}^N$ with $\u^m \in L^2 \big( 0,T; W^{1,2}(\Omega,\mathbb{R}^N ) \big)$ such that
\begin{align*}
\u_i^{m_i} &\rightharpoonup \u^m \quad \text{ weakly in } L^2 \big( 0,T; W^{1,2}(\Omega,\mathbb{R}^N ) \big), \\
\u_i^{m_i} &\rightarrow \u^m \quad \text{ strongly in } L^2_{\loc} \big( 0,T; W_{\loc}^{1,2}(\Omega,\mathbb{R}^N ) \big), \\
\u_i^{m_i} &\rightarrow \u^m \quad \text{ strongly in } L^q \big( 0,T; W^{1,q}(\Omega,\mathbb{R}^N ) \big),
\end{align*}
as $i \longrightarrow \infty$ and for any $q<2$. Further,
\begin{align*}
\u_i^{m_i} &\rightarrow \u^m \quad \text{ strongly in } L^2(\Omega_T,\mathbb R^N), \\
\u_i^{1+m_i} &\rightarrow \u^{1+m} \quad \text{ strongly in } L^{1+\varepsilon}(\Omega_T,\mathbb R^N)
\end{align*}
as $i \longrightarrow \infty$ and for some $\varepsilon >0$.
\end{lemma}

\begin{proof}
In identical fashion to Corollary \ref{kor:weak-convergence}, by the previous uniform energy estimates in Lemma \ref{lem:caccioppoli_global} and the assumptions for $g$ in \eqref{assumption:g} we find a subsequence $(u_i)$ such that $\u_i^{m_i}$ converges weakly to some $\u^m \in L^2 \big( 0,T; W^{1,2}(\Omega,\mathbb{R}^N ) \big)$. By further applying the local result in Theorem \ref{theorem}, we find that, as $i \rightarrow \infty$, $\u_i^{m_i} \rightarrow \u^m$ in the strong sense in $L^2_{\loc} \big( 0,T; W_{\loc}^{1,2}(\Omega,\mathbb{R}^N ) \big)$.

The extension from the local convergence in $L^2_{\loc} \big( 0,T; W_{\loc}^{1,2}(\Omega,\mathbb{R}^N ) \big)$ to convergence in $L^q \big( 0,T; W^{1,q}(\Omega,\mathbb{R}^N ) \big)$ follows by the application of the measure theoretic argument as in the end of the proof for Lemma \ref{lem:strong-lp-convergence}.

For the functions $\u_i^{m_i}$ we apply the Gagliardo-Nirenberg inequality \ref{lem:global-parabolic-sobolev}, once again using that $m_i > m_c$ as well as $m> m_c$ to see that $\u_i^{m_i} \in L^{2+\varepsilon} (\Omega_T, \mathbb R^N)$, for the whole range $m_c < m_i< \infty$. Yet again, we apply the measure theoretic argument to obtain strong convergence $\u_i^{m_i} \rightarrow \u^m$ in $L^2(\Omega_T,\mathbb R^N)$ as $i \rightarrow \infty$.

We turn our attention to the convergence of $u_i$. Define $G_i := 2(1+\frac{1+m_i}{m_in})$ and $G:=  2(1+\frac{1+m}{mn})$ as the exponents appearing in the Gagliardo-Nirenberg inequality from Lemma \ref{lem:global-parabolic-sobolev}. \\
Since $\gamma >1+m$, $G > (1+m)m^{-1}$ and $m_i \rightarrow m$ as $i \rightarrow \infty$ one can find $\varepsilon>0$ such that $\gamma m_i^{-1} > (1+m_i)m_i^{-1}(1+\varepsilon)$ and $G_i > (1+m_i)m_i^{-1}(1+\varepsilon)$ for large $i$. 
Define $q_i:= \min\lbrace \frac \gamma {m_i}, G_i \rbrace$. Then
\begin{align*}
\iint_{\Omega_T} |u_i|^{(1+m_i)(1+\varepsilon)} \d x \d t
	&= \iint_{\Omega_T} |\u_i^{m_i}|^{\frac{1+m_i}{m_i}(1+\varepsilon)} \d x \d t \\
	&\leq  c|\Omega_T| + c\iint_{\Omega_T} |\u_i^{m_i}|^{q_i} \d x \d t \\
	&\leq c+ c \iint_{\Omega_T} |\u_i^{m_i}-\g^{m_i}|^{G_i} \d x \d t
		+ c\iint_{\Omega_T} |g|^{\gamma} \d x \d t
\end{align*}
The second integral is finite by the assumptions for $g$ in \eqref{assumption:g}. For the first term, we use the global version of Gagliardo-Nirenberg to obtain
\begin{align*}
\iint_{\Omega_T} &|\u_i^{m_i}-\g^{m_i}|^{G_i} \d x \d t \\
	&\leq \bigg( \iint_{\Omega_T} |D\u_i^{m_i} - D\g^{m_i}|^2 \d x \d t \bigg)
		\bigg( \sup_{t \in (0,T)} \int_\Omega |\u_i^{m_i}-\g^{m_i}|^{\frac{1+m_i}{m_i}} \d x \bigg)^{\frac 2 n}
\end{align*}
By using the uniform bounds from Lemma \ref{lem:caccioppoli_global} and once again the assumptions for $g$ in \eqref{assumption:g}, the right hand side is uniformly bounded. We can conclude that $\u^{1+m_i}$ forms a bounded sequence in $L^{1+\varepsilon}(\Omega_T, \mathbb R^N)$. By applying the measure theoretic argument (as in the proof for Lemma \ref{lem:strong-lp-convergence}), we can see that $\u_i^{1+m_i}$ converges to $\u^{1+m}$ strongly in the space $L^{1+\varepsilon}(\Omega_T,\mathbb R^N)$ for some $\varepsilon>0$.
\end{proof}

\begin{lemma} \label{lem:global-limitproblem}
The limit function $u$ is a global weak solution to Equation~\eqref{equation:global-PME} attaining the corresponding boundary values $g$ in the sense of Definition~\ref{def:global_weak_solution}. 
\end{lemma}

\begin{proof}
The fact that $u$ satisfies the integral equality follows in the same way as in the local case in Lemma \ref{lem:limit_function_is_sol} by using strong convergences for $u_i$ and $D\u_i^{m_i}$.

Here we need to show that the limit function attains the boundary values as well. By assumption, for every $i$ we have that $\u_i^{m_i} - \g^{m_i} \in L^2(0,T;W^{1,2}_0(\Omega,\R^N))$, which is a closed and convex subset of $L^2(0,T;W^{1,2}(\Omega,\R^N))$. Further, $\u_i^{m_i} - \g^{m_i} \rightharpoonup \u^m - \g^m$ weakly in $L^2(0,T;W^{1,2}(\Omega, \mathbb R^N))$. By the Hahn-Banach theorem it then follows that $\u^m - \g^m \in L^2(0,T;W^{1,2}_0(\Omega,\R^N))$. This implies that the limit function attains the boundary values on the lateral boundary as in~\eqref{lateral-boundary}. \\
For the initial boundary we use the equation. Let us define a cut-off function in time as 
\begin{equation*}
  \alpha (t):=
  \begin{cases}
    0, &t< \delta,\\
    \frac{1}{\delta}(t - \delta), &t\in [\delta, 2\delta], \\
	1, &t\in (2\delta,\tau - h), \\
	\frac{1}{h}(\tau - t), &t\in [\tau- h, \tau], \\
	0, &t > \tau ,
  \end{cases}
\end{equation*}
where $2 \delta < \tau-h$, $\tau < T$ and $\delta, h > 0$. We will test the mollified equation by $\varphi = \alpha (\u_i^{m_i} - \g^{m_i})$. For the divergence part we obtain
\begin{align*}
\iint_{\Omega_T} \alpha \llbracket &\mathbf A(x,t,u_i, D \u_i^{m_i}) \rrbracket_\lambda \cdot D (\u_i^{m_i} - \g^{m_i})   \, \d x \d t \\
&\xrightarrow{\lambda \to 0} \iint_{\Omega_T} \alpha \mathbf A(x,t,u_i, D \u_i^{m_i})  \cdot D (\u_i^{m_i} - \g^{m_i})  \, \d x \d t.
\end{align*}
As in the proof for the Caccioppoli type estimate in Lemma \ref{lem:caccioppoli_global}, leading up to equation \eqref{eq:EnergyEstimate_LHS}, we can estimate the parabolic part as
\begin{align*}
\liminf_{\lambda\to 0} \iint_{\Omega_T} &\alpha \partial_t \llbracket u_i \rrbracket_\lambda (\u_i^{m_i} - \g^{m_i}) \, \d x\d t \\
	&\ge \iint_{\Omega_T} \alpha \partial_t \g^{m_i} (u_i-g)\, \d x \d t 
	- \iint_{\Omega_T} \partial_t \alpha I_i(u_i,g)  \, \d x \d t.
\end{align*}
By combining these estimates, again similarly to Lemma \ref{lem:caccioppoli_global}, we have
\begin{align*}
- \iint_{\Omega_T} \partial_t \alpha &I_i(u_i,g)  \, \d x \d t \\
	&\leq -\iint_{\Omega_T} \alpha \partial_t \g^{m_i} (u_i-g)\, \d x \d t  \\
	&\qquad-\iint_{\Omega_T} \alpha \mathbf A(x,t,u_i, D \u_i^{m_i})  \cdot D (\u_i^{m_i} - \g^{m_i})  \, \d x \d t \\
	&\leq \iint_{\Omega_T} \alpha |\partial_t \g^{m_i} (u_i-g)|\, \d x \d t \\
	&\qquad -\nu \iint_{\Omega_T} \alpha |D\u_i^{m_i}|^2 \d x \d t
		+L \iint_{\Omega_T} \alpha |D\u_i^{m_i}| |D\g^{m_i}| \d x \d t.
\end{align*}
We rearrange terms, apply Young's inequality and absorb the resulting $|D\u_i^{m_i}|^2$ term to the left hand side. This yields
\begin{align*}
- \iint_{\Omega_T} \partial_t &\alpha I_i(u_i,g)  \, \d x \d t + \frac{\nu}{2} \iint_{\Omega_T} \alpha |D\u_i^{m_i}|^2  \, \d x \d t \\
&\le c\iint_{\Omega_T} \alpha |D\g^{m_i}|^2  \, \d x \d t + c\iint_{\Omega_T} \alpha |\partial_t \g^{m_i} (g - u_i)| \, \d x \d t,
\end{align*}
from which we can conclude
%
%
\begin{align*}
\frac{1}{h} &\int_{\tau-h}^{\tau} \int_{\Omega} I_i(u_i,g) \, \d x\d t - \frac{1}{\delta} \int_{\delta}^{2\delta} \int_{\Omega} I_i(u_i,g) \, \d x\d t \\ 
&\le c \int_0^\tau \int_{\Omega} \left( |u_i|^{m_i +1} + |g|^{m_i+1} 
 + |D\g^{m_i}|^2 +  |\partial_t \g^{m_i}|^\frac{m_i+1}{m_i} \right)  \, \d x \d t \\
&\le  c \int_0^\tau \int_{\Omega} \left( |u_i|^{m_i +1} + |g|^p 
 + |D\g^{\tilde m}|^\beta +  |\partial_t \g^{\tilde m}|^\frac{\gamma}{\tilde m} + 1 \right)  \, \d x \d t,
\end{align*}
where $p = \max \{\gamma, \beta \tilde m\}$ by using Lemma~\ref{lem:g_conditions}. First we let $\delta \to 0$, so that the second term on left hand side vanishes by the initial condition \eqref{initial-boundary}. Then by first letting $i \to \infty$, on the left hand side we have
$$
\frac{1}{h} \int_{\tau-h}^{\tau} \int_{\Omega} I_i(u_i,g) \, \d x\d t \longrightarrow
\frac{1}{h} \int_{\tau-h}^{\tau} \int_{\Omega} I(u,g) \, \d x\d t.
$$
This follows from the boundedness of $u$ in $L^{1+m}(\Omega_T,\mathbb R^N)$, as seen in Lemma \ref{lem:caccioppoli_global}, the convergence of $u_i$ in $L^{1+m}(\Omega_T,\mathbb R^N)$ by Lemma \ref{lem:global_convergences_immediate} and from the convergence and boundedness properties of $g$, see \eqref{assumption:g} and Lemma \ref{lem:g_conditions}. \\
Now letting $h \to 0$ results in
$$
\frac{1}{h} \int_{\tau-h}^{\tau} \int_{\Omega} I(u,g) \, \d x\d t
\longrightarrow
\int_\Omega I (u (\cdot,\tau), g (\cdot,\tau))\, \d x.
$$
Eventually by taking the limit as $\tau \to 0$, we have that
$$
\int_\Omega I (u (\cdot,\tau), g (\cdot,\tau))\, \d x \longrightarrow 0.
$$
This implies that $u$ attains also the initial values in sense of~\eqref{initial-boundary}.
\end{proof}

\begin{lemma}
In the setting of Theorem \ref{theorem:global} we have
\begin{align*}
D\u_i^{m_i} &\rightarrow D\u^m \quad \text{ strongly in } L^2 (\Omega_T, \mathbb R^{Nn}),
\end{align*}
as $i \longrightarrow \infty$.
\end{lemma}

\begin{proof}
Let $0< \sigma < T$. Define $\xi=\xi_\sigma(t) \in W^{1,\infty}([0,T],[0,1])$ with $\xi_\sigma(t) = 1$ for $t \in [0,T-\sigma]$ and $\xi_\sigma(t) = (T-t)/\sigma$ for $t \in [T-\sigma,T]$. We can apply the monotonicity condition \eqref{monotone} to obtain
\begin{align*}
\iint_{\Omega \times [0,T-\sigma]} &|D\u_i^{m_i} - D\u^m |^2 \dx \dt \\
	 &\leq \iint_{\Omega_T} \xi |D\u_i^{m_i} - D\u^m |^2 \dx \dt \\
	&\leq c \iint_{\Omega_T} \xi \big(
		\A(x,t,u_i,D\u_i^{m_i}) - \A(x,t,u,D\u^m)
		\big)
		\cdot ( D\u_i^{m_i} - D\u^m ) \dx \dt \\
	&=c \iint_{\Omega_T} \xi \A(x,t,u_i,D\u_i^{m_i}) 
		\cdot ( D\u_i^{m_i} - D\mollifytime{\u^m}h ) \dx \dt \\
		&\quad +c \iint_{\Omega_T} \xi \A(x,t,u_i,D\u_i^{m_i}) 
		\cdot ( D\mollifytime{\u^m}h - D\u^m ) \dx \dt \\
		&\quad - c \iint_{\Omega_T} \xi \A(x,t,u,D\u^m)
		\cdot ( D\u_i^{m_i} - D\u^m ) \dx \dt
	=: \mathrm I + \mathrm{II}- \mathrm {III} .
\end{align*}
Regarding the second term, by Cauchy-Schwartz and the uniform bound from Lemma \ref{lem:caccioppoli_global} it follows that
\begin{align*}
\mathrm{II}
	&\leq c \bigg( \iint_{\Omega_T} |D\u_i^{m_i}|^2 \dx \dt \bigg)^2
	\bigg( \iint_{\Omega_T} |D\mollifytime{\u^m}h - D\u^m|^2 \dx \dt \bigg)^2 \\
	&\leq c \bigg( \iint_{\Omega_T} |D\mollifytime{\u^m}h - D\u^m|^2 \dx \dt \bigg)^2.
\end{align*}
This vanishes as $h \downarrow 0$ by the properties of the time mollification. \\
For the third term on the right hand side, notice that the growth condition \ref{growth} implies $\A(x,t,u,D\u^m) \in L^2(\Omega_T,\mathbb R^{Nn})$. Thus the weak convergence $D\u_i^{m_i} \rightharpoonup D\u^m$ in $L^2(\Omega_T,\mathbb R^{Nn})$ implies that $\mathrm {III}$ vanishes as $i \rightarrow \infty$. \\
For the first term, we must use the equation. We begin by adding and subtracting both $\xi \A(x,t,u_i,D\u_i^{m_i}) D\g^{m_i}$ and $\xi \A(x,t,u_i,D\u_i^{m_i}) D\mollifytime{\g^m}h$. The term containing $D\g^{m_i}-D\mollifytime{\g^m}h$ will vanish due to the uniform bound from Lemma \ref{lem:caccioppoli_global} and the conditions for $g$, see \eqref{assumption:g} and Lemma \ref{lem:g_conditions}. \\
For the other terms, we test the mollified equation \eqref{eq:mollifiedEquation} for $u_i$ against the test function $\varphi = \xi \big( \u_i^{m_i} - \g^{m_i} + \mollifytime{\g^m}h - \mollifytime{\u^m}h \big)$.
The following calculations are similar to the ones in Lemma \ref{lem:strong-w1p-convergence}: The divergence part of the mollified equation \eqref{eq:mollifiedEquation} equals
\begin{align*}
\iint_{\Omega_T} &\xi \mollifytime {\A(x,t,u_i,D\u_i^{m_i})} \lambda
	D \big( \u_i^{m_i} - \g^{m_i} 
		+ \mollifytime{\g^m}h - \mollifytime{\u^m}h \big) \dx\dt \\
	&\quad \longrightarrow 
	\iint_{\Omega_T} \xi \A(x,t,u_i,D\u_i^{m_i})
	D \big( \u_i^{m_i} - \g^{m_i} 
		+ \mollifytime{\g^m}h - \mollifytime{\u^m}h \big) \dx\dt
\end{align*}
as $\lambda \downarrow 0$. The right hand side of the mollified equation \eqref{eq:mollifiedEquation} vanishes, since $\varphi(0)=0$ due to the initial conditions \eqref{initial-boundary} being fulfilled for all $u_i$ by assumption and since the time mollification is zero at $t=0$ by definition.  \\
The remaining parabolic part, after moving it to the other side of the equation, is given by
\begin{align*}
-\iint_{\Omega_T} \xi \partial_t \mollifytime {u_i} \lambda
		\big( \u_i^{m_i} - \g^{m_i}
		+\mollifytime{\g^m}h - \mollifytime{\u^m}h  \big) \dx\dt.
\end{align*}
The term containing the difference $\g^{m_i} - \mollifytime{\g^m}h$ will vanish, again due to the uniform bound from Lemma \ref{lem:caccioppoli_global} and the conditions for $g$. We thus concentrate on the remaining term
\begin{align*}
\iint_{\Omega_T} \xi &\partial_t \mollifytime {u_i} \lambda
		\big( -\u_i^{m_i} + \mollifytime{\u^m}h  \big) \dx\dt
		\leq \iint_{\Omega_T} \xi \partial_t \mollifytime {u_i} \lambda
		\big( -\mollifytime {\u_i}\lambda ^{m_i}+ \mollifytime {\u^m}h  \big) \dx\dt 
		= \mathrm{I}+\mathrm{II}.
\end{align*}
For $\mathrm{I}$ we perform integration by parts. The appearing boundary terms vanish, since the mollification of $u_i$ vanishes at time $t=0$ by definition, while $\xi$ vanishes at $t=T$. This implies
\begin{align}\label{eq:grad_L2_conv_eq}
\begin{aligned}
\mathrm{I}
	&= - \frac{1}{1+m_i} \iint_{\Omega_T}  \xi \partial_t |\mollifytime {u_i} \lambda|^{1+m_i} \dx \dt\\
	&= \frac{1}{1+m_i} \iint_{\Omega_T}  \partial_t\xi |\mollifytime {u_i} \lambda|^{1+m_i} \dx \dt \\
	&
	\longrightarrow
	\frac{1}{1+m} \iint_{\Omega_T}  \partial_t\xi |u|^{1+m} \dx \dt,
\end{aligned}
\end{align}
by first letting $\lambda \downarrow 0$, then $i \rightarrow \infty$.
For the remaining term we perform integration by parts:
\begin{align*}
\mathrm{II}
	&=\iint_{\Omega_T}\xi \partial_t \mollifytime {u_i} \lambda\mollifytime{\u^{m}}h \\
	&= -\iint_{\Omega_T} \partial_t \xi \mollifytime {u_i} \lambda\mollifytime{\u^{m}}h
	+ \xi \mollifytime {u_i} \lambda \partial_t \mollifytime{\u^{m}}h
	\dx \dt \\
	&\stackrel{\lambda \downarrow 0}\longrightarrow
	-\iint_{\Omega_T} \partial_t \xi u_i \mollifytime{\u^{m}}h
	+ \xi u_i \partial_t \mollifytime{\u^{m}}h
	\dx \dt \\	
	&\stackrel{i \rightarrow \infty}\longrightarrow
	-\iint_{\Omega_T} \partial_t \xi u \mollifytime{\u^{m}}h
	+ \xi u \partial_t \mollifytime{\u^{m}}h
	\dx \dt  \\
	&\leq 	-\iint_{\Omega_T} \partial_t \xi u \mollifytime{\u^{m}}h
	+ \xi \mollifytime {\u^m} h^{1/m} \partial_t \mollifytime{\u^{m}}h
	\dx \dt  \\
	&= -\iint_{\Omega_T} \partial_t \xi u \mollifytime{\u^{m}}h
	+ \frac{m}{1+m}\xi \partial_t |\mollifytime{\u^{m}}h|^{\frac 1 m + 1}
	\dx \dt  \\
	&=-\iint_{\Omega_T} \partial_t \xi u \mollifytime{\u^{m}}h 
	- \frac{m}{1+m} \partial_t \xi |\mollifytime{\u^{m}}h|^{\frac 1 m + 1}
	\dx \dt  \\
	&\stackrel{h \downarrow 0}\longrightarrow
	\iint_{\Omega_T} -\partial_t \xi |u|^{1+m}
	+ \frac{m}{1+m} \partial_t \xi |u|^{1+m}
	\dx \dt.
\end{align*}
These terms cancel out together with the term from \eqref{eq:grad_L2_conv_eq}. Collecting these results yields the claim.
\end{proof}

This completes the proof of Theorem \ref{theorem:global}. Again, we proved convergence for a subsequence. In contrast to the local setting, in this case we are not given any weak limit of the original sequence by assumption. Thus we cannot determine the convergence of the original sequence in this generality. However, as mentioned in the Remark~\ref{rem:uniqueness}, in the model case $\A(x,t,u,D\u^m) = D\u^m$ one can show uniqueness of weak solutions as in  V\'azquez \cite[Theorem 5.3]{Vazquez} by using suitable lower bounds for $(\u^m-\v^m)\cdot (u-v)$. In this way one can conclude uniqueness of the limit and thus convergence for the whole sequence.


\begin{thebibliography}{99}

  

\bibitem{Benilan-Boccardo-Herrero}
\newblock Ph.~B\'{e}nilan, L.~Boccardo, and M.~A.~Herrero.
\newblock On the limit of solutions of {$u_t=\Delta u^m$} as
{$m\to\infty$}.
\newblock Some topics in nonlinear PDEs (Turin, 1989). {\em Rend. Sem. Mat. Univ. Politec. Torino 1989, Special Issue}, 1–13 (1991).

\bibitem{Benilan-Crandall}
\newblock Ph.~B\'{e}nilan and M. G. Crandall.
\newblock The continuous dependence on {$\varphi $} of solutions of
{$u_{t}-\Delta \varphi (u)=0$}.
\newblock {\em Indiana Univ. Math. J.} 30 (1981), no. 2, 161–177.



\bibitem{Benilan-Igbida}
\newblock Ph.~B\'{e}nilan and N.~Igbida.
\newblock Singular limit of changing sign solutions of the porous medium
equation.
\newblock {\em J. Evol. Equ.} Vol. 3 (2003), no. 2, 215--224.


\bibitem{Boegelein-Dietrich-Vestberg}
\newblock V.~B\"ogelein, N.~Dietrich, and M.~Vestberg.
\newblock Existence
of solutions to a diffusive shallow medium equation.
\newblock {\em Preprint} (2020).

\bibitem{Boegelein-Duzaar-Gianazza}
\newblock V.~B\"ogelein, F.~Duzaar, and U.~Gianazza.
\newblock Porous medium type equations with measure data and potential estimates.
\newblock {\em SIAM J. Math. Anal.} Vol. 45 (2013), 3283--3330.



\bibitem{Boegelein-Duzaar-Korte-Scheven}
\newblock V.~B\"ogelein, F.~Duzaar, R.~Korte and C.~Scheven.
\newblock The higher integrability of weak solutions of porous medium systems.
\newblock {\em Adv.~Nonlinear Anal.} 8 (2019), no. 1, 1004--1034.


\bibitem{BDM:pq}
\newblock V.~B\"ogelein, F.~Duzaar, and P.~Marcellini.
\newblock Parabolic systems with $p,q$-growth: a variational approach.
\newblock  {\em Arch. Ration. Mech. Anal.} 210(1):219--267, 2013.

\bibitem{Boegelein-Duzaar-Scheven:2018}
\newblock V.~B\"ogelein, F.~Duzaar, and C.~Scheven.
\newblock Higher integrability for the singular porous medium system.
\newblock Preprint 2018.

\bibitem{Boegelein-Lukkari-Scheven}
\newblock V.~B\"ogelein, T.~Lukkari, and C.~Scheven.
\newblock The obstacle problem for the porous medium equation.
\newblock {\em Mathematische Annalen} 363(1):455--499, 2015.


\bibitem{Caffarelli-Friedman}
\newblock L.~A.~Caffarelli and A.~Friedman.
\newblock Asymptotic behavior of solutions of {$u_t=\Delta u^m$} as
{$m\to\infty$}.
\newblock {\em Indiana Univ. Math. J.} 36 (1987), no. 4, 711–728.

\bibitem{Chen-Karlsen}
\newblock G.-Q.~Chen and K.~H.~Karlsen.
\newblock {$L^1$}-framework for continuous dependence and error
estimates for quasilinear anisotropic degenerate parabolic
equations.
\newblock {\em Trans. Amer. Math. Soc.} 358 (2006), no. 3, 937–963.

\bibitem{Cockburn-Gripenberg}
\newblock B.~Cockburn and G.~Gripenberg.
\newblock Continuous dependence on the nonlinearities of
solutions of degenerate parabolic equations.
\newblock {\em J. Differential Equations} 151 (1999), no. 2, 231–251.

\bibitem{DiBe}
E.~DiBenedetto.
\newblock {\em Degenerate parabolic equations}.
\newblock Springer-Verlag, Universitytext xv, 387, New York, NY, 1993.


\bibitem{Dibenedetto-Gianazza-Liao}
\newblock E.~Dibenedetto, U.~Gianazza, and N.~Liao.
\newblock Logarithmically singular parabolic equations as limits of the
porous medium equation.
\newblock {\em Nonlinear Anal.} 75 (2012), no. 12, 4513–4533.



\bibitem{DBGV-book} 
E. DiBenedetto, U. Gianazza, and V. Vespri. 
{\it Harnack's inequality for degenerate and singular parabolic equations}.
Springer Monographs in Mathematics, 2011.










\bibitem{Hui}
\newblock K.~M.~Hui.
\newblock Singular limit of solutions of the equation
{$u_t=\Delta({u^m}/m)$} as {$m\to0$}.
\newblock {\em Pacific J. Math.} 187 (1999), no. 2, 297–316.

\bibitem{Hui-veryfast}
\newblock K.~M.~Hui.
\newblock Singular limit of solutions of the very fast diffusion
equation.
\newblock {\em Nonlinear Anal.} 68 (2008), no. 5, 1120–1147.

\bibitem{Karlsen-Risebro}
\newblock K.~H.~Karlsen and N.~H.~Risebro.
\newblock On the uniqueness and stability of entropy solutions of nonlinear degenerate parabolic equations with rough coefficients.
\newblock {\em Discrete Contin. Dyn. Syst.} 9 (2003), no. 5, 1081–1104.

\bibitem{Kinnunen-Lewis:1}
\newblock J.~Kinnunen and J.~L.~Lewis.
\newblock Higher integrability for parabolic systems of $p$-Laplacian type.
\newblock {\em Duke Math.~J.} 102 (2000), no.~2, 253-271.


\bibitem{Kinnunen-Lindqvist}
\newblock J.~Kinnunen and P.~Lindqvist.
\newblock Pointwise behaviour of semicontinuous supersolutions to a quasilinear parabolic equation.
\newblock {\em Ann. Mat. Pura Appl.} (4) 185(3):411--435, 2006.

\bibitem{Kinnunen-Parviainen}
\newblock J.~Kinnunen and M.~Parviainen.
\newblock Stability for degenerate parabolic equations.
\newblock {\em Adv.~Cal.~Var.} 3 (2010), no. 1, 29--48.


\bibitem{Lindqvist}
\newblock P. Lindqvist.
\newblock Stability for the solutions of {$\mathrm{div}(|\nabla u|^{p-2} \nabla u) = f$} with varying {$p$}.
\newblock {\em J. Math. Anal. App.} 127 (1987), no. 1, 93--102

\bibitem{Lukkari}
\newblock T.~Lukkari.
\newblock Stability of solutions to nonlinear diffusion equations.
\newblock {\em preprint} (2013) 

\bibitem{Lukkari-Parviainen}
\newblock T.~Lukkari and M.~Parviainen.
\newblock Stability of degenerate parabolic Cauchy problems
\newblock {\em Commun. Pure Appl. Anal.} 14 (2015), no. 1, 201-–216. 


\bibitem{Moring-Scheven-Schwarzacher-Singer}
\newblock K.~Moring, C.~Scheven, S.~Schwarzacher, T.~Singer. 
\newblock Global higher integrability of weak solutions of porous medium systems.
\newblock {\em Comm. Pure Appl. Anal.} 19 (2020), no. 3, 1697–1745.

\bibitem{Parviainen}
\newblock M.~Parviainen.
\newblock Global gradient estimates for degenerate parabolic equations in nonsmooth domains. 
\newblock{\em Ann.~Mat.~Pura Appl.}  (4) 188 (2009), no.~2, 333--358.



\bibitem{Simon}
\newblock J.~Simon.
\newblock Compact sets in the space {$L^p(0,T;B)$}.
\newblock {\em Ann. Mat. Pura App.} 146(4):65--96.

\bibitem{Sturm}
\newblock S.~Sturm.
\newblock Existence of weak solutions of doubly nonlinear parabolic
equations.
\newblock {\em J. J. Math. Anal. App.} 455 (2017), 842-863

\bibitem{Singer-Vestberg}
\newblock T.~Singer and M.~Vestberg.
\newblock Local boundedness of weak solutions to the Diffusive
Wave Approximation of the Shallow Water equations.
\newblock {\em J. Differential Equations} 266 (2019), 3014-3033

\bibitem{Vazquez}
\newblock J.~L.~V\'azquez.
\newblock {\em The porous medium equation.}
\newblock Oxford Mathematical Monographs. The Clarendon Press, Oxford University Press, Oxford, 2007.

\bibitem{Vespri-Vestberg}
\newblock V.~Vespri and M.~Vestberg.
\newblock An extensive study of the regularity properties of solutions to doubly singular equations.
\newblock {\em Adv. Calc. Var.} (To appear)

 
\end{thebibliography}
\end{document}